\documentclass{article}
\usepackage[utf8]{inputenc}
\usepackage{amssymb,amsmath,amsfonts,amsthm}
\usepackage{fullpage}
\usepackage{booktabs}
\usepackage{xspace}
\usepackage{framed}
\usepackage[usenames,dvipsnames,svgnames,table]{xcolor}
\usepackage{enumitem}
\usepackage{subcaption}
\usepackage[disableredefinitions]{complexity}
\usepackage{authblk}
\usepackage{thm-restate}
\usepackage{tikz}
\usetikzlibrary{positioning,shapes,fit,calc,decorations.pathmorphing}
\tikzstyle{vert} = [circle,fill=black, minimum size=2mm, inner sep=0pt]
\tikzstyle{edge} = [thick]
\usepackage[colorlinks=true,linkcolor=blue,citecolor=ForestGreen]{hyperref}

\usepackage[capitalise]{cleveref}

\definecolor{col0}{HTML}{c42626}
\definecolor{col1}{HTML}{ffd700}
\definecolor{col2}{HTML}{ff00ba}
\definecolor{col3}{HTML}{00ff00}
\definecolor{col4}{HTML}{0000bb}
\definecolor{col5}{HTML}{1e90ff}
\definecolor{col6}{HTML}{7fffd4}

\definecolor{cola}{HTML}{ff0000}
\definecolor{colb}{HTML}{00ff00}
\definecolor{colc}{HTML}{0000ff}

\newsavebox\ideabox
\newenvironment{idea}
{\begin{equation}
		\begin{lrbox}{\ideabox}
			\begin{minipage}{\dimexpr\columnwidth-2\leftmargini}
				\setlength{\leftmargini}{0pt}%
				\begin{quote}}
				{\end{quote}
			\end{minipage}
		\end{lrbox}\makebox[0pt]{\usebox{\ideabox}}
\end{equation}}

\newtheorem{theorem}{Theorem}[section]
\newtheorem{lemma}[theorem]{Lemma}
\newtheorem{claim}[theorem]{Claim}
\newtheorem{corollary}[theorem]{Corollary}

\newtheorem{proposition}[theorem]{Proposition}
\newtheorem{conjecture}[theorem]{Conjecture}

\newtheorem{remark}[theorem]{Remark}

\newenvironment{poc}{\begin{proof}[Proof of Claim]}{\end{proof}}
\theoremstyle{remark}

\renewcommand{\leq}{\leqslant}
\renewcommand{\geq}{\geqslant}

\newcommand{\trw}{\mathsf{tw}}

\renewcommand{\Pr}[1]{\mathbb{P}\!\left(\,#1\,\right)}

\newcommand{\Ex}[1]{\mathbb{E} \left[\,#1\,\right]}

\newcommand{\eps}{\varepsilon}

\newcommand{\mlcr}{\textsc{Multi-layer cops and robber}\xspace}

\newcommand{\mlcralloc}{\textsc{Allocated multi-layer cops and robber}\xspace}
\newcommand{\mlcrany}{\textsc{Multi-layer cops and robber with free layer choice}\xspace}

\newcommand{\cop}[1]{\mathsf{c}\!\left(#1\right)} 
\newcommand{\mcop}[1]{\mathsf{mc}\!\left(#1\right)}
\newcommand{\maxcop}[2]{ \mathsf{emc}_{#1}\!\left(#2\right)}  

\newcommand{\flatten}[1]{\mathsf{fl}\!\left(#1\right)}  

\newclass{\EXPTIME}{EXPTIME}

\newcommand{\mlg}{\ensuremath{\mathcal{G}}}

\newcommand{\layers}{\ensuremath{\tau}}

\newcommand{\alloc}{\ensuremath{\mathbf{k}}}

\providecommand{\keywords}[1]{\textbf{\textit{Keywords---}} #1}
\providecommand{\codes}[1]{\textbf{\textit{AMS MSC 2010---}} #1}

\title{\textbf{Cops and Robbers on Multi-Layer Graphs}\thanks{An extended abstract of this paper appeared at WG 2023 \cite{EnrightMPS23}.}}

\author[1]{Jessica Enright}
\author[1]{Kitty Meeks}
\author[1]{William Pettersson}
\author[1,2]{John Sylvester}
\affil[1]{School of Computing Science, University of Glasgow, UK\\
	\texttt{\{jessica.enright,kitty.meeks,william.pettersson\}@glasgow.ac.uk}}
\affil[2]{Department of Computer Science, University of Liverpool, UK\\ \texttt{john.sylvester@liverpool.ac.uk}} 

\date{}
\begin{document}

	\maketitle
	\begin{abstract} We generalise the popular \textit{cops and robbers} game to multi-layer graphs, where each cop and the robber are restricted to a single layer (or set of edges). 
		We show that initial intuition about the best way to allocate cops to layers is not always correct, and prove that the multi-layer cop number is neither bounded from above nor below by any increasing function of the cop numbers of the individual layers. We determine that it is $\NP$-hard to decide if $k$ cops are sufficient to catch the robber, even if every cop layer is a tree and a set of isolated vertices. However, we give a polynomial time algorithm to determine if $k$ cops can win when the robber layer is a tree. Additionally, we investigate a question of worst-case divisions of a simple graph into layers: given a simple graph $G$, what is the maximum number of cops required to catch a robber 
		over all multi-layer graphs where each edge of $G$ is in at least one layer and all layers are connected?
		For cliques, suitably dense random graphs, and graphs of bounded treewidth, we determine this parameter up to multiplicative constants.
		Lastly we consider a multi-layer variant of Meyniel's conjecture, and show the existence of an infinite family of graphs whose multi-layer cop number is bounded from below by a constant times $n / \log n$, where $n$ is the number of vertices in the graph. 
	\end{abstract}
	\\
	\codes{05C57,68R10}
	\section{Introduction}
	
	Cops and robbers is a 2-player adversarial game played on a graph introduced independently by Nowakowski and Winkler~\cite{NOWAKOWSKI1983235}, and Quilliot~\cite{quilliot1978jeux}. At the start of the game, the cop player chooses a starting vertex for each of a specified number of cops, and the robber player then selects a starting vertex for the robber.  
	Then in subsequent rounds, the cop player first chooses none, some, or all cops and moves them along exactly one edge to a new vertex.
	The robber player then either moves the robber along an edge or leaves the robber on its current vertex.
	The cop player wins if after some finite number of rounds a cop occupies the same vertex as the robber, and the robber wins otherwise.  Both players have perfect information about the graph and the locations of cops and robbers.  
	Initially, research focussed on games with only one cop and one robber, and graphs on which the cop could win were classed as \emph{copwin} graphs.
	Aigner and Fromme~\cite{AF84} introduced the idea of playing with multiple cops, and defined the \emph{cop number} of a graph as the minimum number of cops required for the cop player to win on that graph. Many variants of the game have been studied, and for an in-depth background on cops and robbers, we direct the reader to~\cite{Bonato2011Book}.
	
	In this paper, we play cops and robbers on multi-layer graphs where each cop and the robber will be associated with exactly one layer, and during their respective turns, will move only along edges that exist in their layer.  While we define multi-layer graphs formally in upcoming sections, roughly speaking here a multi-layer graph is a single set of vertices with each layer being a different (though possibly overlapping) set of edges on those vertices.
	The variants we study could intuitively be based on the premise that the cops are assigned different modes of transport.
	For instance, a cop in a car may be able to move quickly down streets, while a cop on foot may be slower down a street, but be able to quickly cut between streets by moving through buildings or down narrow alleys.
	
	Extending cops and robbers to multi-layer graphs creates some new variants, and generalises some existing ones.
	Fitzpatrick~\cite{FitzpatrickThesis} introduced the \emph{precinct} variant, which assigns to each cop a subset of the vertices (called their \emph{beat}), and each cop can never leave their beat. This can be modelled as multi-layer cops and robbers by restricting each layer to a given beat. Fitzpatrick~\cite{FitzpatrickThesis} mainly considers the case were a beat is an isometric path, we allow more arbitrary (though usually spanning and connected) beats/layers. Clarke~\cite{ClarkeThesis} studies the problem of covering a graph with a number of cop-win subgraphs to upper bound the cop number of a graph --- again such constructions can be modelled as multi-layer graphs with the edges of each layer forming a cop-win graph.
	Another commonly studied variant of cops and robbers defines a speed $s$ (which may be infinite) such that the robber can move along a path of up to $s$ edges on their turn~\cite[Section 3.2]{Bonato2017Book}.
	These can also be modelled as multi-layer graphs by adding edges between any pair of vertices of distance at most~$s$ that only belong to the layer the robber is occupying.
	
	\subsection{Further Related Work} \label{sec:further}
	Temporal graphs, in which edges are active only at certain time steps, are sometimes modelled as multi-layered graphs. There has been some work on cops and robbers on temporal graphs, though generally yielding quite a different game to the ones we consider here as a cop is not restricted to one layer.  In particular, \cite{Balev2020DynamicGraphs} considers cops and robbers on temporal graphs and when the full temporal graph is known they give a $\mathcal{O}(n^3T)$ algorithm to determine the outcome of the game where $T$ is the number of timesteps. See also \cite{CarufelFSS23}.
	
	Variants of cops and robbers are also studied for their relationships to other parameters of graphs. For instance, the cop number of a graph $G$ is at most one plus half the treewidth of $G$~\cite{JoretKT10}.
	And if one considers the ``helicopter'' variant of cops and robbers, the treewidth of a graph is strictly less than the helicopter cop number of the graph~\cite{SeymourThomas}.
	A recent pre-print of Toru\'{n}czyk~\cite{toruczyk2023flipwidth} generalises many graph parameters, including treewidth, clique-width, degeneracy, rank-width, and twin-width, through the use of variants of cops and robbers.
	We introduce our multi-layer variants of cops and robbers partially in the hopes of spurring research towards multi-layer graph parameters using similar techniques.
	
	Recently Lehner, resolving a conjecture by Schr\"oeder \cite{Schroeder}, showed the cop number of a toroidal graph is at most three \cite{Lehner21}. There is also an interesting connection between cop number and the genus of the host graph \cite{AF84,BowlerELP21,Quilliot85a,Schroeder}. It remains open whether any such connection can be made in the multi-layer setting.
	
	\subsection{Outline and Contributions}
	In Section \ref{sec:background-notation} we define multi-layer graphs and multi-layer cops and robbers. 
	
	In Section \ref{sec:counterexamples} we develop several examples which highlight several counter-intuitive facts and properties of the multi-layer cops and robbers game. In particular, we show the multi-layer cop number is not bounded from above or below by a non-trivial function of the cop numbers of the individual cop layers.

	In Section \ref{sec:hardness} we study the computational complexity of some multi-layer cops and robbers problems. We show that deciding if a given number of cops can catch a robber is $\NP$-hard even if every cop layer is a tree and possibly some isolated vertices. On the contrary we show that if only the robber layer is required to be a tree, then the problem is $\FPT$ in the number of cops and the number of layers of the graph. 
	
	In Section \ref{sec:upper-bounds} we consider an extremal version of multi-layer cop number over all divisions into layers of a single-layer graph. In particular, for a given single-layer graph $G$ we consider the maximum multi-layer cop number of any multi-layer graph $\mlg$ when all edges of $G$ are present in at least one layer of $\mlg$. We determine this variant of the cop number upto constants for cliques and suitably dense random graphs, and provide a bound on it by the treewidth of $G$.

	In Section~\ref{sec:men} we consider Meyniel's conjecture, which states that the single-layer cop number of any connected $n$-vertex graph is $\mathcal{O}(\sqrt{n})$ and is a central open question in cops and robbers. We investigate whether a multi-layer analogue of Meyniel's conjecture can hold and give a construction with only two cop layers which requires $\Omega\left(n/ \log n\right)$ cops to police.  This determines the worst case multi-layer cop number up to a multiplicative $\mathcal{O}(\log n)$ factor. This contrasts with the situation on simple graphs, where the worst-case is only known up to a multiplicative $n^{1/2-o(1)}$ factor.
	
	Finally, in Section \ref{sec:conclusion} we reflect and conclude with some open problems. 
	
	\section{Definitions and Notation}\label{sec:background-notation}
	
	We write $[n]$ to mean the set of integers $\{1,\ldots, n\}$.
	\paragraph{Graphs.} Given a set $V$ we write $\binom{V}{2}$ to mean all possible 2-element subsets (i.e., edges) of $V$.
	A simple graph is then defined as $G:=(V,E)$ where $E \subseteq \binom{V}{2}$.
	For a vertex $v\in V$ we let $d_{G}(v) := \left|\{u : uv\in E  \}\right|$ be the degree of vertex $v$ in $G$, and $\delta(G) := \min_{v\in V(G)} d_G(v)$ denote the minimum degree in a graph $G$.
	If, for all $v\in V$, $d_{G}(v) = r$ for some integer $r$, we say that $G$ is $r$-regular. If the exact value of $r$ is not important, we may just say that $G$ is regular. If, instead, for all $v\in V$, $d_{G}(v) \in \{r,r+1\}$ for some integer $r$, we say that $G$ is almost-regular.
	The \emph{distance} between two vertices $u$ and $v$ in a graph is the length of a shortest path between $u$ and $v$.
	
	\paragraph{Multi-Layer Graphs.} A multi-layer graph $(V,\{E_1, \dots, E_\layers\})$ consists of a vertex set $V$ and a collection $\{E_1, \dots, E_\layers\}$, for some integer $\layers\geq 1$, of edge sets (or \textit{layers}), where for each $i$, $E_i\subseteq \binom{V}{2}$.
	We often slightly abuse terminology and refer to a layer $E_i$ as a graph; when we do this, we specifically refer to the graph $(V, E_i)$, i.e., we always include every vertex in the original multi-layer graph, even if such a vertex is isolated in $(V, E_i)$.
	For instance, we often restrict ourselves to multi-layer graphs where, for each $i\in[\layers]$, the simple graph $(V, E_i)$ is connected. In this instance we say that each layer is \emph{connected}.
	Given a multi-layer graph $(V,\{E_1, \dots E_\layers\})$ let the \emph{flattened} version of a multi-layer graph, written as $\flatten{\mlg}$, be the simple graph $G=(V, E_1 \cup \cdots\cup E_\layers)$.
	
	\paragraph{Cops and Robbers on Graphs.} Cops and robbers is typically played on a simple graph, with one player controlling some number of cops and the other player controlling the robber.
	On each turn, the cop player can move none, some, or all of the cops, however each cop can only move along a single edge incident to their current vertex.
	The robber player can then choose to move the robber along one edge, or have the robber stay still.
	The goal for the cop player is to end their turn with the robber on the same vertex as at least one cop, while the aim for the robber is to avoid capture indefinitely.
	If a cop player has a winning strategy on a graph $G$ with $k$ cops but not with $k-1$ cops, we say that the graph $G$ has cop number $k$, denoted $\cop{G} = k$, and that $G$ is $k$-copwin.
	Given a multi-layer graph $(V, \{E_1, \ldots, E_\layers\})$, we will say the cop number of layer $E_i$ to mean the cop number of the graph $(V, E_i)$.
	
	As this paper deals with both simple and multi-layer graphs, as well as cops and robbers variants played on these graphs, we will use \emph{single-layer} as an adjective to denote when we are referring to either specifically a simple graph, or to cops and robbers played on a single-layer (i.e.,~simple) graph.
	This extends to parameters such as the cop number as well.
	
	\paragraph{Cops and Robbers on Graphs on Multi-Layer Graphs.} In this paper we consider the cops and robbers game on multi-layer graphs and so it will be convenient to define multi-layer graphs with a distinguished layer for the robber. More formally, for an integer $\layers \geq 1$, we use the notation $\mlg= (V,\{C_1, \dots, C_\layers\},R)$ to denote a multi-layer graph with vertex set $V$ and collection $\{C_1, \dots, C_\layers,R \}$ of layers, where $\{C_1,\dots , C_\layers\}$ are the cop layers and $R$ is the robber layer. In the cops and robber game on $\mlg$ each cop is allocated to a single-layer from $\{C_1,\dots, C_\layers\} $, and the robber to $R$, and each cop (and the robber) will then only move along edges in their respective layer. We do not allow any cop or the robber to move between layers. This is a slight abuse of notation, and that both $(V,\{C_1,\dots, C_\layers,R\}) $ and $(V,\{C_1,\dots, C_\layers\},R)$ both denote a multi-layer graphs with the same collection $ \{C_1,\dots, C_\layers\}\cup \{R\}  $ of edge sets, the latter has designated layers for the robber/cops whereas the former does not. We will use $E_i$ to denote edge sets in multi-layer graphs that do not have a cop or robber labels.
	
	A setting that appears frequently is $R= C_1\cup \cdots \cup C_\layers$, where the robber can use any edge that exists in a cop layer. This setting is given by the multi-layer graph $\mlg := (V, \{C_1, \ldots, C_\layers\}, C_1 \cup \cdots \cup C_\layers)$, but
	for readability we will instead use $\mlg := (V, \{C_1, \ldots, C_\layers\}, *)$ to denote this. 
	
	We define several variants of cops and robbers on multi-layer graphs, however in each of them we have an \textit{allocation} $\alloc := (k_1,\ldots, k_\layers)$ of cops to layers, such that there are $k_i$ cops on layer $C_i$.
	We will often use $k := \sum_i k_i$ to refer to the total number of cops in a game.
	
	We now define multi-layer cops and robbers: a two player game played with an allocation $\alloc$ on a multi-layer graph $\mlg= (V,\{C_1, \dots, C_\layers\},R) $. The two players are the cop player and the robber player.
	Each cop is assigned a layer such that there are exactly $k_i$ cops in layer $C_i$.
	The game begins with the cop player assigning each cop to some vertex, and then the robber player assigns the robber to some vertex.
	The game then continues with each player taking turns in sequence, beginning with the cop player.
	On the cop player's turn, the cop player may move each cop along one edge in that cop's layer.
	The cop player is allowed to move none, some, or all of the cops.
	The robber player then takes their turn, either moving the robber along one edge in the robber layer or letting the robber stay on its current vertex.
	This game ends as a victory for the cop player if, at any point during the game, the robber is on a vertex that is also occupied by one or more cops.
	The robber wins if they can evade capture indefinitely.
	
	\paragraph{Formal Problem Statements.}  We can now begin defining our problems, starting with \mlcralloc.
	
	\begin{framed}
		\noindent
		\mlcralloc \\
		\textit{Input:} A tuple $(\mlg, \alloc)$ where $\mlg=(V, \{C_1,\ldots,C_\layers\}, R)$ is a multi-layer graph and $\alloc$ is an allocation of cops to layers. \\
		\noindent\textit{Question:} Does the cop player have a winning strategy when playing multi-layer cops and robbers on $\mlg$ with allocation $\alloc$?
	\end{framed}
	
	We also consider a variant in which the cop player has a given number $k$ of cops, but gets to choose the layers to which the cops are allocated.
	
	\begin{framed}
		\noindent
		\mlcr \\
		\textit{Input:} A tuple $(\mlg, k)$ where $\mlg=(V, \{C_1,\ldots,C_\layers\}, R)$ is a multi-layer graph and $k\geq 1$ is an integer. \\
		\noindent\textit{Question:} Is there an allocation $\alloc$ with $\sum_i k_i = k$ such that $(\mlg, \alloc)$ is yes-instance for \mlcralloc?
	\end{framed}

	We say that the multi-layer cop number of a multi-layer graph $\mlg$ is $k$ if $(\mlg, k)$ is a yes-instance for \mlcr but $(\mlg, k-1)$ is a no-instance for \mlcr.
	We will denote this with $\mcop{\mlg}$.

	Lastly we consider \mlcrany, a variant of \mlcr in which, before the game is played, the layers in the multi-layer graph are not assigned to being either cop layers or robber layers.
	Instead, the layers are simply labelled $E_1$ through $E_\layers$, and 
	in this variant the cop player first allocates each cop to one layer, and then the robber player is free to allocate the robber to any layer. 
	
	\begin{framed}
		\noindent
		\mlcrany \\
		\textit{Input:} A tuple $(\mlg, k)$ where $\mlg=(V,\{E_1, \dots, E_\layers\} )$ is a multi-layer graph and $k\geq 1$ is an integer. \\
		\noindent\textit{Question:} Is there an allocation $\alloc$ with $\sum_i k_i = k$ such that for every $j$, $((V, \{E_1, \ldots, E_\layers\}, E_j), \alloc)$ is a yes-instance for \mlcr?
	\end{framed}

	\paragraph{Basic Relations.} We round out this section with a number of basic observations. 
	
	\begin{proposition}\label{prop:subsetrobber}
		Let $\mlg=(V,\{C_1, \dots, C_\layers\} , R)$ and $\mlg'=(V,\{C_1, \dots, C_\layers\} ,R')$ be any two multi-layer graphs where $R\subseteq R'\subseteq \binom{V}{2}$.
		If $(\mlg, k)$ is a no-instance to \mlcr, then $(\mlg', k)$ is a no-instance to \mlcr. Consequently, $\mcop{\mlg} \leq \mcop{\mlg'}$.
	\end{proposition}
	\begin{proof}
		To win, the robber on $\mlg'$ uses the strategy from $\mlg$.
		The robber can execute this strategy as any edge in $R'$ is in $R$.
		Since the cop layers have no added edges, the strategy must be robber-win as else the cops would win on $\mlg$.
	\end{proof}
	
	\begin{proposition}\label{prop:subsetcops}
		Let $\mlg=(V,\{C_1, \dots, C_\layers\} , R)$ and $\mlg'=(V,\{C'_1, \dots,  C'_\layers\} ,R)$ be any two multi-layer graphs that satisfy $C_i\subseteq C_i'$ for every $i\in[\layers]$. If $(\mlg, k)$ is a yes-instance to \mlcr, then $(\mlg', k)$ is also a yes-instance to \mlcr. 
	\end{proposition}
	\begin{proof}
		To win, the cops on $\mlg'$ use the strategy from $\mlg$.
		As no edge has been removed from $\mlg$ to create $\mlg'$, this must still result in the cops winning.
	\end{proof}
	
	\begin{proposition}\label{prop:global-to-free-layer}
		Let $\mlg = (V, \{C_1, \ldots, C_\layers\}, *)$ be a multi-layer graph. If $(\mlg, k)$ is a yes-instance for \mlcr, then, letting $E_i = C_i$ for each $i\in[\layers]$, $((V, \{E_1,\ldots, E_\layers\}), k)$ is a yes-instance for \mlcrany.
	\end{proposition}
	\begin{proof}
		This immediately follows from the problem definitions and Proposition~\ref{prop:subsetrobber}.
	\end{proof}
	
	\section{Counter Examples \& Anti-Monotonicity Results}\label{sec:counterexamples}
	
	In this section we provide some concrete examples of cops and robbers on multi-layer graphs illustrating some peculiarities of the game that may seem counter-intuitive.
	In particular, we show that it is sometimes beneficial to put multiple cops on the same layer, and leave other layers empty.
	We then show that there is no lower bound on the cop number of a multi-layer graph that is a function only of the cop numbers of the individual cop layers.
	In the other direction, we also show that there is no upper bound on the cop number of a multi-layer graph that is a function of only the cop numbers of the layers.

	\subsection{Two-Layer Grid}\label{sec:two-layer-grid}
	
	In this section we give a multi-layer graph with two connected cop layers and one robber layer,
	where there is an automorphism of the graph that swaps the two cop layers.
	With such symmetric cop layers, an intuitive approach may be to assign one cop to each of the two cop layers.
	However, the follow theorem shows that two cops allocated in such a manner cannot catch the robber, while if both cops are assigned to the same layer,
	they can catch the robber.

	\begin{restatable}{theorem}{gridexist}\label{thm:gridexist}
		For any $n\geq 4$ there exists a multi-layer graph $(V, \{C_H, C_V\}, *)$ on $n$ vertices such that a cop player can win with two cops if both cops are on $C_H$, or if both cops are on $C_V$, but the robber player can win if one cop is on $C_V$ and the other is on $C_H$.
	\end{restatable}
	
	This theorem follows directly from  Lemmas~\ref{lemma:grid-cop-win} and \ref{lemma:grid-robber-win}, however first we must give our construction.

	Let $n$ be a positive integer, and let $\mathcal{G}^{2}_n := (V, \{C_H, C_V\}, R)$ be a multi-layer graph on vertices $V:=\{(i,j)\mid i,j\in \{1,\ldots,n\}\}$ 
	where $C_H$ contains edges 
	\begin{align*}
		&\{(i,j)(i,j+1) \mid i\in\{1,\ldots,n\}, j\in\{1,\ldots,n-1\}\} \\ 
		\cup &\{(i,1)(i+1,1) \mid i \in \{1,\ldots,n-1\} \wedge i \equiv 0 \mod 2 \} \\
		\cup &\{(i,n)(i+1,n) \mid i \in \{1,\ldots,n-1\} \wedge i \equiv 1 \mod 2 \},
	\end{align*}
	$C_V$ contains edges
	\begin{align*}
		&\{(i,j)(i+1,j) \mid j\in\{1,\ldots,n\}, i\in\{1,\ldots,n-1\}\} \\
		\cup &\{(1,j)(1,j+1) \mid j \in \{1,\ldots,n-1\} \wedge j \equiv 0 \mod 2 \} \\
		\cup &\{(n,j)(n,j+1) \mid j \in \{1,\ldots,n-1\} \wedge j \equiv 1 \mod 2 \},
	\end{align*}
	and $R$ contains all edges in either $C_V$ or $C_H$.
	The $H$ in $C_H$ (respectively $V$ in $C_V$) refers to the layer containing all horizontal (respectively vertical) edges, where vertex $(i,j)$ is  interpreted as the $i$th row and $j$th column.
	Some edges around the sides of the grid are contained in both $C_H$ and $C_V$ to allow transitions between rows (or columns), hence $C_H$ contains all horizontal edges and also some vertical edges (and vice-versa for $C_V$).
	See Figure~\ref{fig:grid_graph} for a diagram.

	\begin{figure}[!t]
		\centering
		\begin{tikzpicture}[xscale=.9,yscale=.9]
			\foreach \i in {0,...,5} {
				\foreach \j in {0,...,5} {
					\node[fill,circle,minimum size=.02cm] (\i\j) at (\i,\j) {}; 
				}
			}
			\foreach \i [count=\xi from 2] in {1,...,4} {
				\foreach \j [count=\ji from 1] in {0,...,4} {
					\draw[edge,red,line width=4pt] (\i\j) -- (\i\ji);
					\draw[edge,blue,line width=4pt] (\j\i) -- (\ji\i);
				}
			}
			\node  at (-0.5, -0.5) {$(6,1)$};
			\node at (-0.5, 5.5) {$(1,1)$};
			\node at (5.5, -0.5) {$(6,6)$};
			\node at (5.5, 5.5) {$(1,6)$};
			\draw[edge,red,line width=4pt] (00) -- (01);
			\draw[edge,red,line width=4pt] (02) -- (03);
			\draw[edge,red,line width=4pt] (04) -- (05);
			\draw[edge,red,line width=4pt] (51) -- (52);
			\draw[edge,red,line width=4pt] (53) -- (54);
			
			\draw[edge,blue,line width=4pt] (10) -- (20);
			\draw[edge,blue,line width=4pt] (30) -- (40);
			\draw[edge,blue,line width=4pt] (05) -- (15);
			\draw[edge,blue,line width=4pt] (25) -- (35);
			\draw[edge,blue,line width=4pt] (45) -- (55);
			
			\draw[edge,blue,line width=4pt] (10) -- (00);
			\draw[edge,red,dashed,line width=4pt] (10) -- (00);
			\draw[edge,blue,line width=4pt] (20) -- (30);
			\draw[edge,red,dashed,line width=4pt] (20) -- (30);
			\draw[edge,blue,line width=4pt] (40) -- (50);
			\draw[edge,red,dashed,line width=4pt] (40) -- (50);
			
			\draw[edge,blue,line width=4pt] (01) -- (02);
			\draw[edge,red,dashed,line width=4pt] (01) -- (02);
			\draw[edge,blue,line width=4pt] (03) -- (04);
			\draw[edge,red,dashed,line width=4pt] (03) -- (04);
			
			\draw[edge,blue,line width=4pt] (15) -- (25);
			\draw[edge,red,dashed,line width=4pt] (15) -- (25);
			\draw[edge,blue,line width=4pt] (35) -- (45);
			\draw[edge,red,dashed,line width=4pt] (35) -- (45);
			
			\draw[edge,blue,line width=4pt] (50) -- (51);
			\draw[edge,red,dashed,line width=4pt] (50) -- (51);
			\draw[edge,blue,line width=4pt] (52) -- (53);
			\draw[edge,red,dashed,line width=4pt] (52) -- (53);
			\draw[edge,blue,line width=4pt] (54) -- (55);
			\draw[edge,red,dashed,line width=4pt] (54) -- (55);
		\end{tikzpicture}
		\caption{An example of $\mathcal{G}^2_6$ with the edges of $C_V$ and $C_H$ shown in red and blue respectively. The edges dashed in both colours belong to both $C_V$ and $C_H$.}
		\label{fig:grid_graph}
	\end{figure}
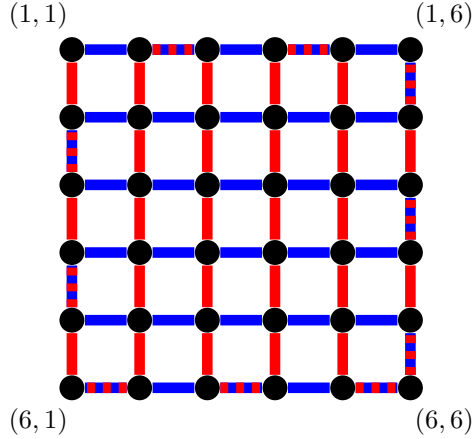
	
	We now consider the game of $\mlcralloc$ with two cops where both of the cops are on the same layer.
	As each individual row or column forms a geodesic path, the cop player can win using a path-guarding strategy on these geodesics: we now give a complete proof of this.
	
	\begin{lemma}\label{lemma:grid-cop-win}
		On $\mathcal{G}^2_n$, a cop player can win with two cops if both cops are on $C_H$ or both cops are on $C_V$.
	\end{lemma}
	
	\begin{proof}
		Without loss of generality, we can use the symmetry of $\mathcal{G}^2_n$ to assume that both cops are on $C_V$.
		Additionally, if $n\leq 2$ then the result holds trivially so assume $n\geq 3$.
		We now proceed by giving the strategy for the cops.
		The two cops start on vertices $(1,1)$ and $(1,2)$, and let $(i,j)$ be the vertex the robber starts on.
		The strategy for the cops have  two phases. In the first phase both cops move to be in the same row as the robber. In the second phase, the cop that is initially closer to the robber always moves to stay in the same row as the robber, but otherwise stays in their column (say column $i$), while the second cop advances until this second cop is in the same row as the robber and is in column $i+1$.
		The cops then repeat this second phase to catch the robber.
		
		{\bf Phase 1}
		In the first phase, assume that the robber ends their turn on row $j$, and that the cops start on $(i,1)$ and $(i,2)$ respectively where $i \leq j$. This trivially holds at the start of the game.
		If $i=j$, both cops stay on $(i,1)$ and $(i,2)$ respectively.
		Otherwise, the cops move from $(i,1)$ and $(i,2)$ to $(i+1,1)$ and $(i+1,2)$ respectively until they are in the same row as the robber after their turn (i.e., before the robber moves).
		The cop player transitions to the second phase of their strategy if, just before the robber moves, the robber is at some position $(i,j)$, the cops are in positions $(i,1)$ and $(i,2)$, and the robber is not yet caught.
		Otherwise, this first phase repeats.
		Note that if the robber is in either of the first two columns at the end of this phase, they are caught and the game ends as a win for the cops.
		
		{\bf Phase 2}
		This phase may be repeated; each iteration starts with the robber taking their move.
		Let the robber be at position $(i,j)$ before their move, let the cops be called $c_m$ (moving cop) and $c_b$ (blocking cop) such that $c_m$ is at position $(i,k)$ and $c_b$ is at $(i,(k+1))$ where $k+1 < j$, and let it be the robber's turn to move.
		Note that Phase 1 guarantees that the game is in such a state when Phase 2 starts.
		If the robber moves to a different row, then $c_b$ moves so that $c_b$ is also in the same row as the robber. As a result, the robber can never enter the same column as $c_b$, that is, the robber can never enter column $k+1$, and thus can also not enter any columns to the left of $k+1$.
		While $c_b$ does this, $c_m$ moves along the path in $C_V$ towards the robber until $c_m$ is in column $k+2$ and is in the same row as the robber.
		Cop $c_m$ can reach such a position in at most $3n$ turns; it suffices for our proof that this can be done in a finite amount of time.
		If the robber is now caught, the game is over. If the robber is not caught, it must be in some column $j' \geq k+3$, while the cops are in columns $k+1$ and $k+2$, and Phase 2 repeats with a reversal of the roles $c_b$ and $c_m$.
		
		In each iteration of Phase 2, the cops move one column further to the right, and block the robber from moving to the columns to the left of the cops, so eventually there will be no $j'$ that satisfies $j' \geq k+3$ and $j'\leq n$, at which point the robber must be caught.
	\end{proof}

	While we have just shown that two cops on the same layer can catch the robber, the following result shows that two cops on different layers cannot catch the robber if $n\geq 4$, highlighting that ``spreading out'' the cops is not necessarily a good strategy for the cop player.
	For $n\in \{1,2\}$, two cops can win trivially, regardless of the layer they are on. For $n=3$, two cops on different layers can always catch a robber; this can be proven by a case analysis.
	
	\begin{lemma}\label{lemma:grid-robber-win}
		On $G^2_n$ a robber player can win against two cops if $n\geq 4$ and the cops are on different layers.
	\end{lemma}
	\begin{proof}
		Denote the two cops by $c_H$ and $c_V$ such that $c_H$ is on the layer defined by $C_H$ and $c_V$ is on the layer defined by $C_V$.
		The robber will win by staying within the vertex set $\{(1,1),(1,2),(2,1),(2,2)\}$,
		and choosing a safe path based on whether $c_H$ is in the first row and whether $c_V$ is in the first column.
		
		Each time the robber is to move, and when the robber is to be placed on the graph at the start of the game, the robber player calculates a vertex $(a,b)$ as follows.
		Let $a = 2$ if $c_H$ is in the first row, and let $a = 1$ otherwise.
		Similarly, let $b = 2$ if $c_V$ is in the first column, and let $b = 1$ otherwise.
		The robber starts the game on $(a,b)$.
		In particular, note that $c_V$ does not start in the same column as the robber, and that $c_H$ does not start in the same row as the robber.
		Thus, after the first cop move, if $c_V$ is on the same column as the robber then $c_V$ is on row $n$,
		and if $c_H$ is on the same row as the robber than $c_H$ is in column $n$.
		
		We now explain the robber strategy for either choosing an edge to move along or choosing to stay still.
		For this, we say that the \emph{cop location restrictions} hold (after the cops have moved) 
		if $c_H$ is on the same row as the robber then $c_H$ is either on column $n$ or column $n-1$, 
		if $c_V$ is on the same column as the robber then $c_V$ is either on row $n$ or row $n-1$, and
		if $c_H$ is on the same row as the robber and $c_V$ is on the same column as the robber then either $c_H$ is on column $n$ or $c_V$ is on row $n$.
		These hold before the first robber move by the previous paragraph,
		and the robber strategy will assume that before the robber moves these cop restrictions hold, and then show that after the robber's turn, and any possible subsequent move by the cop player, 
		these cop location restrictions will hold.
		
		To do this, denote the robber location as $(a',b')$, and recall that $(a,b)$ is calculated as described earlier before every robber move.
		If $a = a'$ and $b = b'$, the robber does not move.
		If $a = a'$ and $b \neq b'$, the robber can move along the one edge from $(a',b')$ to $(a,b)$.
		Similarly, if $a \neq a'$ and $b = b'$, the robber can move along the one edge from $(a',b')$ to $(a,b)$.
		In all of these, we see that for any cop movement, the cop location restrictions still hold, so the robber can repeat these moves indefinitely.
		
		It remains to consider the case where $a \neq a'$ and $b \neq b'$.
		It must be that $c_H$ is on the same row as the robber, and $c_V$ is on the same column as the robber.
		If $c_H$ is on column $n$, the robber moves to $(a',b)$.  
		As the robber changed columns, for $c_V$ to be on the same row as the robber after the cop move, they must end in row $n$, and $c_H$ can be in either column $n$ or column $n-1$, so the cop location restrictions hold.
		If $c_H$ is on column $n-1$, $c_V$ must be on row $n$, and the robber moves to $(a, b')$. 
		As the robber changed rows, for $c_H$ to be on the same row as the robber after the cop move, they must end in row $n$, and $c_V$ can be in either column $n$ or column $n-1$, so the cop location restrictions hold.
		
		Thus, the robber can repeat this strategy indefinitely to win. 
	\end{proof}

	\subsection{No Lower Bound by the Cop-Numbers of Individual Layers}

	It is natural to ask if, for any $n$-vertex multi-layer graph $\mlg = (V, \{C_1, \ldots, C_\layers\}, R)$, the multi-cop number of $\mlg$ is bounded from below by the minimum cop-number of a single cop layer; namely, does $\mcop{\mlg} \geq \min_i \cop{(V, C_i)}$ hold? Observe that, if we let $S_n$ denote the star graph on $n$ vertices, any multi-layer graph $\mlg = (V, \{E(S_n), C_2, \ldots, C_\layers\}, R)$ has cop number $1$, as the cop can start on the centre of the star and reach any other vertex in one move.
	This does not resolve the question directly; however, in the next result we build on this idea to show a general bound of the form  $\mcop{\mlg} = \Omega( \min_i \cop{(V, C_i)})$ does not hold. 
	
	\begin{proposition}\label{prop:mincounter}
		For any $c\geq 2$ there exist graphs $G_1 = (V,E_1)$ and $G_2 = (V, E_2)$  such that $\cop{G_1},\cop{G_2}\geq c$ and $\mcop{(V, \{E_1, E_2\}, *)} = 2$.
	\end{proposition}
	To prove this we shall need the following two lemmas. 
	\begin{lemma}[{\cite[Lemma 3]{AF84}}]\label{lem:mindegbdd}Let $G$ be any connected graph that contains no $3$ or $4$ cycles. Then $\cop{G}\geq \delta(G)$. 
	\end{lemma}
	
	We define a \textit{homomorphism} $\varphi:G\rightarrow H$ as a map $\varphi:V(G)\rightarrow V(H)$ such that for all $x,y\in V(G)$,  $xy\in E(G)$ implies that $\varphi(x)\varphi(y)\in E(H)$. In particular, homomorphisms do not increase distances.
	If there is a homomorphism from $G$ to $H$ that is the identity on $H$, then $H$ is a \emph{retract} of $G$.
	\begin{lemma}[{\cite[Theorem 3.1]{BerInt}}] \label{lem:retract}If $G$ is connected and  $\varphi:G\rightarrow H$  is a retract, then $\cop{H}\leq \cop{G}$.  
	\end{lemma}
	
	We can now prove Proposition \ref{prop:mincounter}, see Figure \ref{fig:arblargemulti} for an illustration of the construction. 
	
	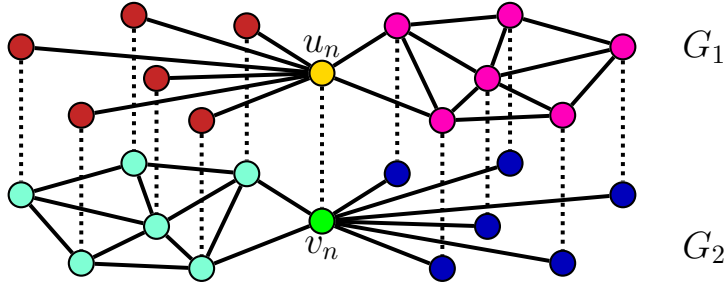
\begin{figure}[!t]
		\centering
		\begin{tikzpicture}[node/.style={thick,circle,draw=black,minimum size=.1cm,fill=white},0node/.style={thick,circle,draw=black,minimum size=.1cm,fill=col0},1node/.style={thick,circle,draw=black,minimum size=.1cm,fill=col1},2node/.style={thick,circle,draw=black,minimum size=.1cm,fill=col2},3node/.style={thick,circle,draw=black,minimum size=.1cm,fill=col3},4node/.style={thick,circle,draw=black,minimum size=.1cm,fill=col6},5node/.style={thick,circle,draw=black,minimum size=.1cm,fill=col4},edge/.style={line width=1.5pt,black},dotedge/.style={line width=1.5pt,black,dotted}, yscale=.7]
			\node[1node] (c) at (0,0) {} ;
			\node[2node] (1) at (1,.9) {} ;
			\node[2node] (2) at (2.5,1.1) {} ;
			\node[2node] (6) at (2.2,-.1) {} ;
			\node[2node] (3) at (1.6,-.9) {} ;
			\node[2node] (4) at (3.2,-.8) {} ;
			\node[2node] (5) at (4,.5) {} ;
			\node[0node] (-1) at (-1,.9) {} ;
			\node[0node] (-2) at (-2.5,1.1) {} ;
			\node[0node] (-6) at (-2.2,-.1) {} ;
			\node[0node] (-3) at (-1.6,-.9) {} ;
			\node[0node] (-4) at (-3.2,-.8) {} ;
			\node[0node] (-5) at (-4,.5) {} ;
			\draw (0,.9) node[anchor=north]{{\Large $u_n$}};

			\draw[edge] (c) to (1);			
			\draw[edge] (c) to (3);
			\draw[edge] (3) to (1);			
			\draw[edge] (1) to (2);
			\draw[edge] (4) to (5);
			\draw[edge] (2) to (5);
			\draw[edge] (6) to (5);
			\draw[edge] (6) to (1);
			\draw[edge] (6) to (3);
			\draw[edge] (6) to (4);
			\draw[edge] (4) to (3);
			\draw[edge] (6) to (2);
			\draw[edge] (c) to (-1);
			\draw[edge] (c) to (-2);
			\draw[edge] (c) to (-3);
			\draw[edge] (c) to (-4);
			\draw[edge] (c) to (-5);
			\draw[edge] (c) to (-6);
			
			\begin{scope}[yshift=-80]
				\node[3node] (c') at (0,0) {} ;
				\node[4node] (1') at (-1,.9) {} ;
				\node[4node] (2') at (-2.5,1.1) {} ;
				\node[4node] (6') at (-2.2,-.1) {} ;
				\node[4node] (3') at (-1.6,-.9) {} ;
				\node[4node] (4') at (-3.2,-.8) {} ;
				\node[4node] (5') at (-4,.5) {} ;
				\node[5node] (-1') at (1,.9) {} ;
				\node[5node] (-2') at (2.5,1.1) {} ;
				\node[5node] (-6') at (2.2,-.1) {} ;
				\node[5node] (-3') at (1.6,-.9) {} ;
				\node[5node] (-4') at (3.2,-.8) {} ;
				\node[5node] (-5') at (4,.5) {} ;
				\draw (5.5,-.5) node[anchor=east]{{\Large $G_2$}};
				\draw (0,-.9) node[anchor=south]{{\Large $v_n$}};
			\end{scope}
			
			\draw[edge] (c') to (1');			
			\draw[edge] (c') to (3');
			\draw[edge] (3') to (1');			
			\draw[edge] (1') to (2');
			\draw[edge] (4') to (5');
			\draw[edge] (2') to (5');
			\draw[edge] (6') to (5');
			\draw[edge] (6') to (1');
			\draw[edge] (6') to (3');
			\draw[edge] (6') to (4');
			\draw[edge] (4') to (3');
			\draw[edge] (6') to (2');
			\draw[edge] (c') to (-1');
			\draw[edge] (c') to (-2');
			\draw[edge] (c') to (-3');
			\draw[edge] (c') to (-4');
			\draw[edge] (c') to (-5');
			\draw[edge] (c') to (-6');

			\draw[dotedge] (c) to (c');
			\draw[dotedge] (1) to (-1');		
			\draw[dotedge] (2) to (-2');
			\draw[dotedge] (3) to (-3');	
			\draw[dotedge] (4) to (-4');
			\draw[dotedge] (5) to (-5');		
			\draw[dotedge] (6) to (-6');
			\draw[dotedge] (-1) to (1');		
			\draw[dotedge] (-2) to (2');
			\draw[dotedge] (-3) to (3');	
			\draw[dotedge] (-4) to (4');
			\draw[dotedge] (-5) to (5');		
			\draw[dotedge] (-6) to (6');
			
			\draw (5.5,.5) node[anchor=east]{{\Large $G_1$}};
			
		\end{tikzpicture}\caption{Illustration of the construction in the proof of Proposition \ref{prop:mincounter}. Note that the two layers are drawn separated; the dotted edges signify that the two end points of that edge are actually the same vertex.}\label{fig:arblargemulti}
	\end{figure}
	\begin{proof}[Proof of Proposition \ref{prop:mincounter}]
		To begin, for any $c\geq 2$, there exists an $n$ and an $n$-vertex graph $H$ of minimum degree at least $c$ and girth at least five\footnote{Consider for instance a $(c,5)$ cage graph~\cite{Exoo2011}}.
		Let $H$ be one such graph on vertices $u_1,\ldots, u_n$.
		It follows by Lemma \ref{lem:mindegbdd} that $\cop{H}\geq c$. 
		We will now construct a multi-layer graph on a vertex set $V$ of size $2n-1$ vertices labelled $v_1,\ldots, v_{2n-1}$. With two layers given by 
		\begin{align*}E_1 &= \{ v_iv_j \mid u_iu_j \in E(H) \} \cup \{v_nv_i \mid i \in \{n+1,\ldots, 2n-1\}\},\\
			E_2 &= \{ v_{2n-i}v_{2n-j} \mid u_iu_j \in E(H) \} \cup \{v_nv_i \mid i \in [n-1]\}.\end{align*}
		Thus, $E_1$ contains a copy of the edges of $H$ transposed to the vertex set $\{v_1,\ldots,v_n\}$ plus $n-1$ pendant edges all incident to $v_n$, and similarly $E_2$ contains the edges of $H$ on $\{v_n,\ldots,v_{2n-1}\}$ plus $n-1$ pendant edges all incident to $v_n$.
		We then define the multi-layer graph $\mlg = (V,\{E_1, E_2\}, *)$;
		see Figure \ref{fig:arblargemulti} for an illustration of this construction.          
		Observe that removing the $n-1$ pendent edges from $(V,E_1)$ (or $(V,E_2$) acts as a sequence of retracts from $G_1$ to $G$, thus $\cop{H}\leq \cop{(V,E_1)}$ by Lemma \ref{lem:retract}, and so $\cop{(V,E_1)},\cop{(V,E_2)}\geq c $.

		To see that $\mlg$ has cop number at most $2$, we place one cop in each layer on vertex $v_n$.
		Clearly the robber cannot reside at vertex $v_n$.
		If the robber is at any vertex $v_1, \dots, v_{n-1}$ then the cop on $E_1$ is one step away from the robber, and if the robber is on any of the vertices $v_{n+1}, \dots, v_{2n-1}$ then the cop on $E_2$ is one step away. So regardless of where the robber starts, two cops can catch the robber after at most one move.
		
		If there is only one cop, without loss of generality we can assume the cop is on layer $E_1$. 
		This then reduces to playing cops and robbers on $(V, E_1)$ and as $c(V,E_1) \geq c \geq 2$, the robber can use their strategy on $(V,E_1)$ to also win on $\mlg$, and so $\mcop{\mlg} = 2$.
	\end{proof}

	Note that as the above proof shows that two cops will catch the robber after at most one cop move, we have the following corollary.
	\begin{corollary}
		For any $c\geq 2$ there exist graphs $G_1 = (V,E_1)$ and $G_2 = (V, E_2)$ such that $\cop{G_1},\cop{G_2}\geq c$, and for any set of edges $R\subseteq \binom{V}{2}$, $\mcop{(V, \{E_1, E_2\}, R)} \leq 2$.
	\end{corollary}

	\subsection{No Upper Bound by Cop Numbers of Individual Layers}\label{sec:cop-number-from-layers} 
	
	The previous section showed that there is no lower bound on the cop number of a multi-layer graph that depends only on the cop numbers of individual layers.
	We now consider the reverse inequality: is the multi-layer cop number bounded from above by a function of the cop numbers of the individual layers?  
	
	If, in a multi-layer graph $\mlg=(V,\{C_1,\dots, C_\layers\},R)$, the robber layer is a subset of one of the cop layers, i.e.\ $R\subset C_i$ for some $i\in [\layers]$, then $\mcop{\mlg}\leq \cop{(V,C_i)}$ as the cop player can allocate $\cop{(V,C_i)}$ cops to layer $i$, ignoring all other cop layers.
	The same reasoning gives an upper bound of $\sum_{i\in [\layers]} \cop{(V,C_i)}$ on the cop number in the ``free choice layer'' variant of the game. Thus, in this special case an upper bound that depends only on the cop numbers of individual layers does exist.
	However, the next result shows that this is not the case in general.   
	\begin{theorem}\label{thm:no-upper-bound-by-cop-number-of-layers}
		For any positive integer $k$, there exists a multi-layer graph $\mlg  = (V, \{C_1, C_2\}, R)$ on $\mathcal{O}(k^3)$ vertices such that:
		\begin{itemize}
			\item each of $(V,R)$, $(V, C_1)$, and $(V, C_2)$ are connected,
			\item $\cop{(V,R)} \leq 3$,
			\item $\cop{(V,C_i)} \leq 2$ for $i\in\{1,2\}$, and
			\item $\mcop{\mlg} \geq k$.
		\end{itemize}
	\end{theorem}
	
	This result is proved by giving the construction of such a graph where, for ease of reading, we have not tried to optimise the number of vertices used. The proof is then split over Claims~\ref{claim:unbound-cop-layer-bounded-cop}, \ref{claim:unbound-robber-layer-bounded-cop}, and \ref{claim:unbound-multi-layer-cop-number}.

	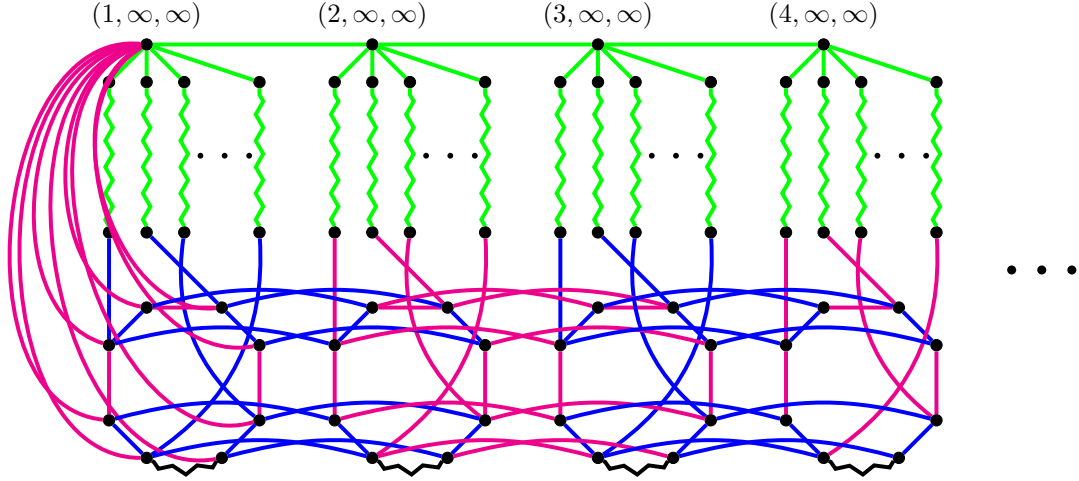
\begin{figure}[!t]
		\centering
		\begin{tikzpicture}[node/.style={circle,draw=black,fill=black,minimum size=1.5mm,inner sep=0pt},edge/.style={line width=1.5pt,black},
			gr/.style={green}, g2/.style={blue}, g3/.style={magenta},
			manymissing/.style={decorate,decoration={zigzag,segment length=3.5mm, amplitude=.5mm},solid}
			]
			
			\node[node,label=above:{$(1,\infty,\infty)$}] (i1) at (1,9) {};
			\node[node] (i1p1) at (0.5, 8.5) {};
			\node[node] (i1p2) at (1.0, 8.5) {};
			\node[node] (i1p3) at (1.5, 8.5) {};
			\node[node] (i1p4) at (2.5, 8.5) {};
			
			\node[node] (i1p1e) at (0.5, 6.5) {};
			\node[node] (i1p2e) at (1.0, 6.5) {};
			\node[node] (i1p3e) at (1.5, 6.5) {};
			\node[node] (i1p4e) at (2.5, 6.5) {};
			
			\node at (2.1, 7.5) {\huge $\cdots$};
			
			\draw[edge,gr] (i1) -- (i1p1);
			\draw[edge,gr] (i1) -- (i1p2);
			\draw[edge,gr] (i1) -- (i1p3);
			\draw[edge,gr] (i1) -- (i1p4);
			
			\draw[edge,gr,manymissing] (i1p1e) -- (i1p1);
			\draw[edge,gr,manymissing] (i1p2e) -- (i1p2);
			\draw[edge,gr,manymissing] (i1p3e) -- (i1p3);
			\draw[edge,gr,manymissing] (i1p4e) -- (i1p4);
			
			\node[node] (i1s1) at (0.5,5) {};
			\node[node] (i1s2) at (1,5.5) {};
			\node[node] (i1s3) at (2,5.5) {};
			\node[node] (i1s4) at (2.5,5) {};
			\node[node] (i1s5) at (2.5,4) {};
			\node[node] (i1s6) at (2,3.5) {};
			\node[node] (i1s7) at (1,3.5) {};
			\node[node] (i1s8) at (0.5,4) {};
			
			\draw[edge,g2] (i1p1e) -- (i1s1);
			\draw[edge,g2] (i1p2e) -- (i1s3);
			\draw[edge,g2] (i1p3e) to[bend right] (i1s5);
			
			\draw[edge,g2] (i1p4e) to[bend left] (i1s7);
			
			\draw[edge,g2] (i1s1) -- (i1s2);
			\draw[edge,g3] (i1s2) -- (i1s3);
			\draw[edge,g2] (i1s3) -- (i1s4);
			\draw[edge,g3] (i1s4) -- (i1s5);
			\draw[edge,g2] (i1s5) -- (i1s6);
			\draw[edge,manymissing] (i1s6) to[bend left] (i1s7);
			\draw[edge,g2] (i1s7) -- (i1s8);
			\draw[edge,g3] (i1s8) -- (i1s1);
			
			\draw[edge,g3] (i1) to[bend right=80] (i1s1);
			\draw[edge,g3] (i1) to[bend right=80] (i1s2);
			\draw[edge,g3] (i1) to[bend right=80] (i1s3);
			\draw[edge,g3] (i1) to[bend right=80] (i1s4);
			\draw[edge,g3] (i1) to[bend right=90] (i1s5);
			\draw[edge,g3] (i1) to[bend right=90] (i1s6);
			\draw[edge,g3] (i1) to[bend right=90] (i1s7);
			\draw[edge,g3] (i1) to[bend right=90] (i1s8);

			\node[node,label=above:{$(2,\infty,\infty)$}] (i2) at (4,9) {};
			\draw[edge,gr] (i1) -- (i2);
			
			\node[node] (i2p1) at (3.5, 8.5) {};
			\node[node] (i2p2) at (4.0, 8.5) {};
			\node[node] (i2p3) at (4.5, 8.5) {};
			\node[node] (i2p4) at (5.5, 8.5) {};
			
			\node[node] (i2p1e) at (3.5, 6.5) {};
			\node[node] (i2p2e) at (4.0, 6.5) {};
			\node[node] (i2p3e) at (4.5, 6.5) {};
			\node[node] (i2p4e) at (5.5, 6.5) {};
			
			\node at (5.1, 7.5) {\huge $\cdots$};
			
			\draw[edge,gr] (i2) -- (i2p1);
			\draw[edge,gr] (i2) -- (i2p2);
			\draw[edge,gr] (i2) -- (i2p3);
			\draw[edge,gr] (i2) -- (i2p4);
			
			\draw[edge,gr,manymissing] (i2p1e) -- (i2p1);
			\draw[edge,gr,manymissing] (i2p2e) -- (i2p2);
			\draw[edge,gr,manymissing] (i2p3e) -- (i2p3);
			\draw[edge,gr,manymissing] (i2p4e) -- (i2p4);
			
			\node[node] (i2s1) at (3.5,5) {};
			\node[node] (i2s2) at (4,5.5) {};
			\node[node] (i2s3) at (5,5.5) {};
			\node[node] (i2s4) at (5.5,5) {};
			\node[node] (i2s5) at (5.5,4) {};
			\node[node] (i2s6) at (5,3.5) {};
			\node[node] (i2s7) at (4,3.5) {};
			\node[node] (i2s8) at (3.5,4) {};

			\draw[edge,g2] (i2s1) -- (i2s2);
			\draw[edge,g3] (i2s2) -- (i2s3);
			\draw[edge,g2] (i2s3) -- (i2s4);
			\draw[edge,g3] (i2s4) -- (i2s5);
			\draw[edge,g2] (i2s5) -- (i2s6);
			\draw[edge,manymissing] (i2s6) to[bend left] (i2s7);
			\draw[edge,g2] (i2s7) -- (i2s8);
			\draw[edge,g3] (i2s8) -- (i2s1);
			
			\draw[edge,g3] (i2p1e) -- (i2s1);
			\draw[edge,g3] (i2p2e) -- (i2s3);
			\draw[edge,g3] (i2p3e) to[bend right] (i2s5);
			\draw[edge,g3] (i2p4e) to[bend left] (i2s7);
			
			\draw[edge,g2] (i1s1) to[bend left=15] (i2s1);
			\draw[edge,g2] (i1s2) to[bend left=15] (i2s2);
			\draw[edge,g2] (i1s3) to[bend left=15] (i2s3);
			\draw[edge,g2] (i1s4) to[bend left=15] (i2s4);
			\draw[edge,g2] (i1s5) to[bend left=15] (i2s5);
			\draw[edge,g2] (i1s6) to[bend left=15] (i2s6);
			\draw[edge,g2] (i1s7) to[bend left=15] (i2s7);
			\draw[edge,g2] (i1s8) to[bend left=15] (i2s8);

			\node[node,label=above:{$(3,\infty,\infty)$}] (i3) at (7,9) {};
			\draw[edge,gr] (i2) -- (i3);
			
			\node[node] (i3p1) at (6.5, 8.5) {};
			\node[node] (i3p2) at (7.0, 8.5) {};
			\node[node] (i3p3) at (7.5, 8.5) {};
			\node[node] (i3p4) at (8.5, 8.5) {};
			
			\node[node] (i3p1e) at (6.5, 6.5) {};
			\node[node] (i3p2e) at (7.0, 6.5) {};
			\node[node] (i3p3e) at (7.5, 6.5) {};
			\node[node] (i3p4e) at (8.5, 6.5) {};
			
			\node at (8.1, 7.5) {\huge $\cdots$};
			
			\draw[edge,gr] (i3) -- (i3p1);
			\draw[edge,gr] (i3) -- (i3p2);
			\draw[edge,gr] (i3) -- (i3p3);
			\draw[edge,gr] (i3) -- (i3p4);
			
			\draw[edge,gr,manymissing] (i3p1e) -- (i3p1);
			\draw[edge,gr,manymissing] (i3p2e) -- (i3p2);
			\draw[edge,gr,manymissing] (i3p3e) -- (i3p3);
			\draw[edge,gr,manymissing] (i3p4e) -- (i3p4);
			
			\node[node] (i3s1) at (6.5,5) {};
			\node[node] (i3s2) at (7,5.5) {};
			\node[node] (i3s3) at (8,5.5) {};
			\node[node] (i3s4) at (8.5,5) {};
			\node[node] (i3s5) at (8.5,4) {};
			\node[node] (i3s6) at (8,3.5) {};
			\node[node] (i3s7) at (7,3.5) {};
			\node[node] (i3s8) at (6.5,4) {};

			\draw[edge,g2] (i3s1) -- (i3s2);
			\draw[edge,g3] (i3s2) -- (i3s3);
			\draw[edge,g2] (i3s3) -- (i3s4);
			\draw[edge,g3] (i3s4) -- (i3s5);
			\draw[edge,g2] (i3s5) -- (i3s6);
			\draw[edge,manymissing] (i3s6) to[bend left] (i3s7);
			\draw[edge,g2] (i3s7) -- (i3s8);
			\draw[edge,g3] (i3s8) -- (i3s1);
			
			\draw[edge,g2] (i3p1e) -- (i3s1);
			\draw[edge,g2] (i3p2e) -- (i3s3);
			\draw[edge,g2] (i3p3e) to[bend right] (i3s5);
			\draw[edge,g2] (i3p4e) to[bend left] (i3s7);
			
			\draw[edge,g3] (i2s1) to[bend left=15] (i3s1);
			\draw[edge,g3] (i2s2) to[bend left=15] (i3s2);
			\draw[edge,g3] (i2s3) to[bend left=15] (i3s3);
			\draw[edge,g3] (i2s4) to[bend left=15] (i3s4);
			\draw[edge,g3] (i2s5) to[bend left=15] (i3s5);
			\draw[edge,g3] (i2s6) to[bend left=15] (i3s6);
			\draw[edge,g3] (i2s7) to[bend left=15] (i3s7);
			\draw[edge,g3] (i2s8) to[bend left=15] (i3s8);
			
			\node[node,label=above:{$(4,\infty,\infty)$}] (i4) at (10,9) {};
			\draw[edge,gr] (i3) -- (i4);
			
			\node[node] (i4p1) at (9.5, 8.5) {};
			\node[node] (i4p2) at (10.0, 8.5) {};
			\node[node] (i4p3) at (10.5, 8.5) {};
			\node[node] (i4p4) at (11.5, 8.5) {};
			
			\node[node] (i4p1e) at (9.5, 6.5) {};
			\node[node] (i4p2e) at (10.0, 6.5) {};
			\node[node] (i4p3e) at (10.5, 6.5) {};
			\node[node] (i4p4e) at (11.5, 6.5) {};
			
			\node at (11.1, 7.5) {\huge $\cdots$};
			
			\draw[edge,gr] (i4) -- (i4p1);
			\draw[edge,gr] (i4) -- (i4p2);
			\draw[edge,gr] (i4) -- (i4p3);
			\draw[edge,gr] (i4) -- (i4p4);
			
			\draw[edge,gr,manymissing] (i4p1e) -- (i4p1);
			\draw[edge,gr,manymissing] (i4p2e) -- (i4p2);
			\draw[edge,gr,manymissing] (i4p3e) -- (i4p3);
			\draw[edge,gr,manymissing] (i4p4e) -- (i4p4);
			
			\node[node] (i4s1) at (9.5,5) {};
			\node[node] (i4s2) at (10,5.5) {};
			\node[node] (i4s3) at (11,5.5) {};
			\node[node] (i4s4) at (11.5,5) {};
			\node[node] (i4s5) at (11.5,4) {};
			\node[node] (i4s6) at (11,3.5) {};
			\node[node] (i4s7) at (10,3.5) {};
			\node[node] (i4s8) at (9.5,4) {};
			
			\draw[edge,g2] (i4s1) -- (i4s2);
			\draw[edge,g3] (i4s2) -- (i4s3);
			\draw[edge,g2] (i4s3) -- (i4s4);
			\draw[edge,g3] (i4s4) -- (i4s5);
			\draw[edge,g2] (i4s5) -- (i4s6);
			\draw[edge,manymissing] (i4s6) to[bend left] (i4s7);
			\draw[edge,g2] (i4s7) -- (i4s8);
			\draw[edge,g3] (i4s8) -- (i4s1);
			
			\draw[edge,g3] (i4p1e) -- (i4s1);
			\draw[edge,g3] (i4p2e) -- (i4s3);
			\draw[edge,g3] (i4p3e) to[bend right] (i4s5);
			\draw[edge,g3] (i4p4e) to[bend left] (i4s7);
			
			\draw[edge,g2] (i3s1) to[bend left=15] (i4s1);
			\draw[edge,g2] (i3s2) to[bend left=15] (i4s2);
			\draw[edge,g2] (i3s3) to[bend left=15] (i4s3);
			\draw[edge,g2] (i3s4) to[bend left=15] (i4s4);
			\draw[edge,g2] (i3s5) to[bend left=15] (i4s5);
			\draw[edge,g2] (i3s6) to[bend left=15] (i4s6);
			\draw[edge,g2] (i3s7) to[bend left=15] (i4s7);
			\draw[edge,g2] (i3s8) to[bend left=15] (i4s8);
			
			\node[fill=black,circle,inner sep=0mm,minimum size=1.2mm] at (12.5,6) {};
			\node[fill=black,circle,inner sep=0mm,minimum size=1.2mm] at (12.9,6) {};
			\node[fill=black,circle,inner sep=0mm,minimum size=1.2mm] at (13.3,6) {};
			
		\end{tikzpicture}
		\caption{A diagram of the construction of $\mlg$ for Theorem~\ref{thm:no-upper-bound-by-cop-number-of-layers}.
			Edges in green are in both $C_1$ and $C_2$, while blue edges are only in $C_1$ and pink edges are only in $C_2$.
			The vertices along the zig-zag green lines are of the form $(x,y,z)$ where $y\in[k]$ and $z\in[5k]$.
			The black zig-zag lines near the bottom of the diagram represent some number of edges alternating in pink and blue, with further pink and blue edges to the next and previous slice.
			The vertices around the cycle formed by the pink and blue edges (including the ones represented by the black zig-zag lines) are of the form $(x,y,z)$ where $y\in[k]$ and $z\in \{5k+1, 5k+2\}$.
		}
		\label{fig:unbounded-cop-number}
	\end{figure}
	We begin by defining the vertex set $V := \{(x, y, z) \mid x\in[3k],\, (y,z) \in \{(\infty, \infty)\} \cup [k] \times [5k+2]\}$.
	For a given $x$, we will refer to the induced subgraph on vertices of the form $(x, y, z)$ (i.e., the induced subgraph on all vertices with a common value for $x$) as the $x$-th slice.
	We will use $(V,E)_x$ to denote the $x$-th slice of $(V,E)$ for some edge set (or layer) $E$.
	
	We now define the edge sets $C_1$ and $C_2$, noting that we will let $R := C_1 \cup C_2$.
	First, for each $y\in [k]$ and each $z\in [5k-1]$ add the edge $(x, y, z)(x, y, z+1)$
	to both $C_1$ and $C_2$.
	Then, for each $y\in [k]$ add the edge $(x, y,5k+1)(x, y,5k+2)$ to $C_1$, and add the edge $(x,y,5k+1)(x,y+1 \mod k, 5k+2)$ to $C_2$.
	Now, for each $x \equiv 1 \mod 2$, and for each $y\in[k]$, add edges of the form $(x,y,5k)(x,y,5k+1)$ to $C_1$,
	and for each $x \equiv 0 \mod 2$, and for each $y\in[k]$, add edges of the form $(x,y,5k)(x,y,5k+1)$ to $C_2$.
	
	We now add edges between different slices.
	For each $x\in [3k-1]$, add the edge $(x,\infty,\infty),(x+1,\infty,\infty)$ to both $C_1$ and $C_2$.
	Then for each $x \equiv 1 \mod 2$ and for each $y\in [k]$ and $z\in [2]$ add the edge $(x,y,5k+z)(x+1,y,5k+z)$ to $C_1$, and
	for each $x \equiv 0 \mod 2$ and for each $y\in [k]$ and $z\in [2]$ add the edge $(x,y,5k+z)(x+1,y,5k+z)$ to $C_2$.
	
	Finally, as $(V,C_2)$ is not currently connected, add edges $(1,\infty,\infty),(1,y,5k+z)$ for each $y\in [k]$ and $z\in[2]$ to $C_2$.
	
	A diagram of this whole construction can be seen in Figure~\ref{fig:unbounded-cop-number}, and a diagram of the induced subgraph on one slice can be seen in Figure~\ref{fig:unbounded-cop-number:one-slice}.
	
	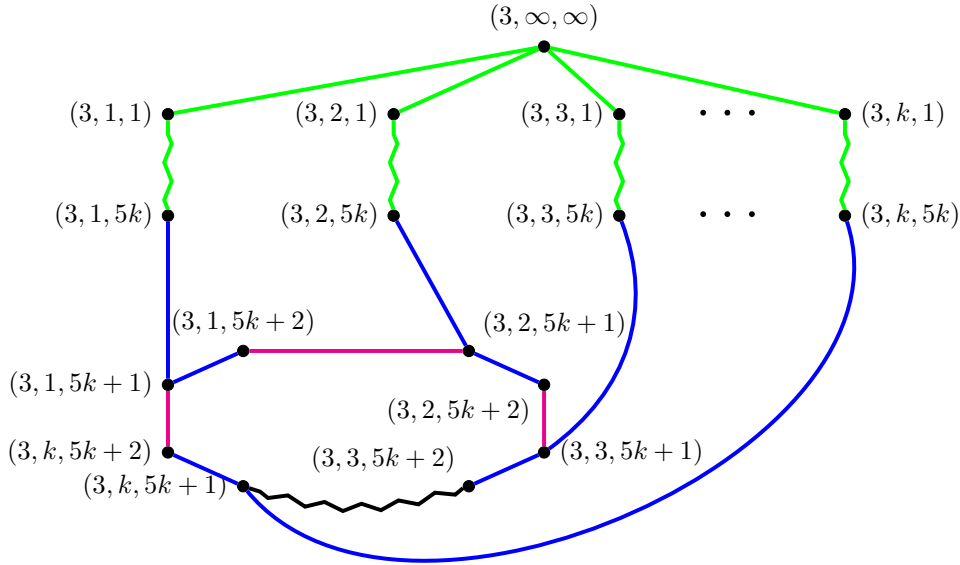
\begin{figure}[!t]
		\centering
		\begin{tikzpicture}[node/.style={circle,draw=black,fill=black,minimum size=1.5mm,inner sep=0pt},edge/.style={line width=1.5pt,black},
			gr/.style={green}, g2/.style={blue}, g3/.style={magenta},
			manymissing/.style={decorate,decoration={zigzag,segment length=5mm, amplitude=.5mm},solid}, yscale=.9	]
			
			\node[node,label=above:{$(3,\infty,\infty)$}] (i1) at (5,9) {};
			\node[node,label=left:{$(3, 1, 1)$}] (i1p1) at (0, 8) {};
			\node[node,label=left:{$(3, 2, 1)$}] (i1p2) at (3, 8) {};
			\node[node,label=left:{$(3, 3, 1)$}] (i1p3) at (6, 8) {};
			\node[node,label=right:{$(3, k, 1)$}] (i1p4) at (9, 8) {};
			
			\node at (7.5, 8) {\huge $\cdots$};
			
			\node[node,label=left:{$(3, 1, 5k)$}] (i1p1e) at (0, 6.5) {};
			\node[node,label=left:{$(3, 2, 5k)$}] (i1p2e) at (3, 6.5) {};
			\node[node,label=left:{$(3, 3, 5k)$}] (i1p3e) at (6, 6.5) {};
			\node[node,label=right:{$(3, k, 5k)$}] (i1p4e) at (9, 6.5) {};
			
			\node at (7.5, 6.5) {\huge $\cdots$};
			
			\draw[edge,gr] (i1) -- (i1p1);
			\draw[edge,gr] (i1) -- (i1p2);
			\draw[edge,gr] (i1) -- (i1p3);
			\draw[edge,gr] (i1) -- (i1p4);
			
			\draw[edge,gr,manymissing] (i1p1e) -- (i1p1);
			\draw[edge,gr,manymissing] (i1p2e) -- (i1p2);
			\draw[edge,gr,manymissing] (i1p3e) -- (i1p3);
			\draw[edge,gr,manymissing] (i1p4e) -- (i1p4);
			
			\node[node,label=left:{$(3,1,5k+1)$}] (i1s1) at (0,4) {};
			\node[node,label=above:{$(3,1,5k+2)$}] (i1s2) at (1,4.5) {};
			\node[node,label=above right:{$(3,2,5k+1)$}] (i1s3) at (4,4.5) {};
			\node[node,label=below left:{$(3,2,5k+2)$}] (i1s4) at (5,4) {};
			\node[node,label=right:{$(3,3,5k+1)$}] (i1s5) at (5,3) {};
			\node[node,label=above left:{$(3,3,5k+2)$}] (i1s6) at (4,2.5) {};
			\node[node,label=left:{$(3,k,5k+1)$}] (i1s7) at (1,2.5) {};
			\node[node,label=left:{$(3,k,5k+2)$}] (i1s8) at (0,3) {};
			
			\draw[edge,g2] (i1p1e) -- (i1s1);
			\draw[edge,g2] (i1p2e) -- (i1s3);
			\draw[edge,g2] (i1p3e) to[bend left=35] (i1s5);
			\draw[edge,g2] (i1p4e) to[bend left=80] (i1s7);
			
			\draw[edge,g2] (i1s1) -- (i1s2);
			\draw[edge,g3] (i1s2) -- (i1s3);
			\draw[edge,g2] (i1s3) -- (i1s4);
			\draw[edge,g3] (i1s4) -- (i1s5);
			\draw[edge,g2] (i1s5) -- (i1s6);
			\draw[edge,manymissing] (i1s6) to[bend left=20](i1s7);
			\draw[edge,g2] (i1s7) -- (i1s8);
			\draw[edge,g3] (i1s8) -- (i1s1);
			
		\end{tikzpicture}    \vspace{-40pt}
		\caption{A diagram of slice 3 of the construction (i.e., the induced subgraph on vertices of the form $(3,y,z)$), with vertices labelled.
			Edges coloured green are in both $C_1$ and $C_2$, while blue edges are only in $C_1$ and pink edges are only in $C_2$.
			Each green zig-zag line represents a path of $5k-1$ green edges.
			The black zig-zag line represents a path of $2(k-4)+1$ alternating pink and blue edges, continuing the shown pattern.
			Every second vertex along this path has a corresponding blue edge to a green path, continuing the pattern of green paths.
			Note that this diagram does not indicate where edges between slices would be present.
		}
		\label{fig:unbounded-cop-number:one-slice}
	\end{figure}
	
	\newpage
	
	We now show that when each layer is considered in conjunction with $V$ as a simple graph, said simple graph has a constant cop number.
	
	\begin{claim}\label{claim:unbound-cop-layer-bounded-cop}
		For $i\in\{1,2\}$,  $\cop{(V,C_i)} \leq 2$.
	\end{claim}
	\begin{proof}
		We prove this claim by giving the cop strategy for two cops.
		The cops both start on $(1,\infty,\infty)$.
		
		First, some extra details if we are considering the graph $(V,C_2)$.
		If the robber starts in the first slice, one cop can stay on $(1,\infty,\infty)$, and be adjacent to any vertex of the form $(1,y,5k+z)$ for $j\in[k]$ and $z\in[2]$.
		The remaining safe vertices reachable by the robber without going through $(1, \infty, \infty)$ form a collection of vertex-disjoint paths in the first slice, and there is no way for the robber to leave the first slice without traversing $(1,\infty,\infty)$. So while one cop stays on $(1,\infty,\infty)$ the other cop traverses the path containing the robber to catch the robber, and thus the cops will win.
		We can therefore assume that the robber does not start in the first slice, which means we can proceed by giving the strategy for two cops on $(V, C_1)$. By symmetry, the same strategy will work for two cops on $(V, C_2)$, giving the result.
		
		During the course of this game, let $(x^R,\infty, \infty)$ be the vertex of the form $(x,\infty, \infty)$ that is closest to the robber.
		By construction, there is always a unique such vertex.
		
		The cops move along vertices of the form $(x,\infty,\infty)$ towards $(x^R, \infty, \infty)$ on each of their moves, noting that the value of $x^R$ may change during these moves.
		Note that if the cops are on $(x,\infty,\infty)$ and the robber is not yet caught, we must have $x^R \geq x$, so the cops must eventually reach $(x^R,\infty,\infty)$.
		Once the cops reach $(x^R,\infty,\infty)$, as it is a cut vertex, the robber is constrained to a subgraph of $(V,C_1)$.
		This subgraph is either a path on $5k$ vertices, or is isomorphic to a path on the vertices $(1,\ldots, 5k+4)$ along with the extra edge $(5k+1,5k+4)$.
		In either case, it is straight-forward to see that the two cops will catch the robber on this subgraph, hence two cops can win.
	\end{proof}
	
	We now find an upper bound for the cop number of one slice in $(V,R)$. We use this result later to upper bound the cop number of $(V,R)$.
	
	\begin{claim}\label{claim:unbound-robber-layer-oneslice}
		For any $i\in[3k]$, the cop number of slice $i$ in $(V,R)$ is at most 2.
	\end{claim}
	\begin{proof}
		First, some extra details if we are playing on the first slice.
		One cop can start on, and stay on, $(1,\infty,\infty)$, and be adjacent to any vertex of the form $(1,y,5k+z)$ for $j\in[k]$ and $z\in[2]$.
		The other cop starts on $(1,1,5k+1)$.
		The safe vertices for the robber form a collection of vertex-disjoint paths in this slice, so while the one cop stays on $(1,\infty,\infty)$ the other cop traverses the path containing the robber to catch the robber, and thus the cops will win.
		
		We now consider any other slice. These are all isomorphic, so without loss of generality, we can assume we are playing on the second slice.
		The two cops start on $(2,1,5k+1)$ and $(2,1,5k+2)$, and assume the robber starts on $(2,y^R,z^R)$.
		Let $(y',z') = (1,5k+2)$ if $y^R = z^R = \infty$, otherwise let $y' = y^R$ and let $z' = 5k+1$ if $z^R < 5k+1$ and let $z' = z^R$ otherwise.
		The cop on $(2,1,5k+2)$ starts moving towards $(2,y', z')$. As this may take multiple turns, the robber may move during this, meaning that the values of $y'$ and $z'$ need to be recalculated.
		However, if the robber moves to a vertex $(2,y^R,z^R)$ where $z^R\not\in\{5k+1,5k+2\}$, the values of $y'$ and $z'$ do not change and so the chasing cop gets closer to $(2,y',z')$ in such a turn.
		Additionally, vertices of the form $(2,y,z)$ where $y\in[k]$ and $z\in \{5k+1,5k+2\}$ induce a cycle in the slice, and the second cop waiting on $(2,1,5k+2)$ means that the robber cannot indefinitely move along this cycle without getting caught by one of the two cops.
		Therefore, the robber must eventually start moving up a path towards $(2,\infty,\infty)$, meaning the chasing cop must eventually reach $(2,y',z')$.
		
		Once the chasing cop can start their turn on $(2,y',z')$, the next phase starts.
		At this point, by the definition of $y'$ and $z'$ the robber is either on $(2,\infty,\infty)$, or on $(2,y',z^R)$ where $z^R\in[5k]$.
		The chasing cop then continues along the path induced by vertices of the form $(2,y',z)$ for $z\in[5k]$, and then to $(2,\infty,\infty)$.
		This takes at most $5k$ turns.
		As these turns pass, if the robber ever moves to $(2,y',\lfloor 5k/2\rfloor )$ for some $y'$, the second cop (that was still on $(2,1,5k+1)$) starts moving towards $(2,y', 5k+1)$.
		This second cop will be able to reach $(2,y^R,5k+1)$ before the robber, meaning the robber is stuck on the induced subgraph formed by vertices in $\{(2,\infty,\infty)\} \cup 2 \times [k] \times [5k]$.
		However, this subgraph is a tree, so the first cop can catch the robber on this subgraph, so the cop number of any given slice is at most 2.
	\end{proof}
	
	Next, we show that the robber layer in this construction, when treated as a simple graph, has a cop number of at most three.
	We show this using the following result, where for two graphs $G_1 := (V,E_1)$ and $G_2 := (V, E_2)$ on the same vertex set, $G_1 - G_2 := (V, E_1 \setminus E_2)$.
	\begin{lemma}[{\cite[Theorem 3.2]{BerInt}}] \label{lem:retract-subgraph}If $G_1$ is connected and  $\varphi:G_1 \rightarrow G_2$  is a retract, then $\cop{G_1}\leq \max\{ \cop{G_2}, \cop{G_1-G_2} + 1\}$.  
	\end{lemma}
	We can then apply this inductively to prove the following claim. 
	\begin{claim}\label{claim:unbound-robber-layer-bounded-cop}
		$\cop{(V,R)} \leq 3$.
	\end{claim}
	
	\begin{proof}
		To begin, for any $x' \in [3k]$ let $(V,R)^\leq_{x'}$ be the induced subgraph of $(V,R)$ on vertices $(x,y,z)$ where $i\leq x'$ (i.e., the induced subgraph on the first $x'$ slices).
		Then noting that $(V,R)^\leq_{x'}$ is a retract of $(V,R)^\leq_{x'+1}$ for $x' \in [3k-1]$, a first application of
		Lemma~\ref{lem:retract-subgraph} on $(V,R)_2$ gives
		\[\cop{(V,R)} \leq \max \left\{ \mathsf{c}\big((V,R)^\leq_{3k-1}\big), \cop{ (V,R)_{3k} } + 1 \right\}.\]
		Further inductive applications of Lemma~\ref{lem:retract-subgraph}, with $G_1 = (V,R)^\leq_i$ and $G_2 = (V,R)_i$ for $i\in [3k-1]$ descending, leads to
		$\cop{(V,R)} \leq \max \{ \cop{(V,R)_2}, \cop{(V,R)_1} + 1 \} = 3.$
	\end{proof}
	
	We now show that if there are less than $k$ cops on $\mlg$, then the robber has a winning strategy, and thus the multi-layer cop number of $\mlg$ is at least $k$.
	\begin{claim}\label{claim:unbound-multi-layer-cop-number}
		$\mcop{\mlg} \geq k$.
	\end{claim}
	\begin{proof}
		We will show that the robber can always avoid capture by $k-1$ cops, and will do so by only staying on vertices of the form $(x,y,z)$ where $z\in\{5k+1,5k+2\}$.
		We say that the \emph{cop location restrictions} are satisfied
		if, when the robber is on vertex $(x^R,y^R,z^R)$ where $z^R\in\{5k+1,5k+2\}$, no cop is within $k$ moves of any vertex of the form $(x^R,y,z)$ where $y\in[k]$ and $z\in\{5k+1,5k+2\}$.
		
		As there are $3k$ slices, the robber can always find a slice $x^R$ such that there is no cop on any of slices $x^R-1$, $x^R$, or $x^R+1$.
		The robber then starts on $(x^R,1, 5k+1)$, so that the cop location restrictions are satisfied.
		
		Now, assuming that the cop location restrictions are satisfied, we show that there is some safe vertex $(x^s,y^s,z^s)$ such that the robber has a path to $(x^s,y^s,z^s)$, no cop can catch the robber while the robber moves to $(x^s,y^s,z^s)$, and regardless of where the cops move, the cop location restrictions will again be satisfied when the robber reaches $(x^s,y^s,z^s)$.
		The robber player, on their turn, chooses an $x^s$ such that there is no cop on any of slices $x^s-1$, $x^s$, or $x^s+1$ (noting that if $x^s-1=0$ or $x^s+1 > 3k$ then vacuously no cop is on those slices).
		Next, note that for any cop on vertex $(x^c,y^c,z^c)$ on layer $i$, the set $\left\{ y \mid \exists x\in[3k], \exists z\in\{5k+1,5k+2\} \text{ s.t. } d_i((x^c,y^c,z^c), (x,y,z)) < 5k\right\}$ has at most two elements (i.e., for any cop, there are at most two values of $y$ such that this cop can reach some vertex $(x,y,z)$ with $x\in[3k]$ and $z\in\{5k+1,5k+2\}$ in less than $5k$ turns).
		Thus, as we have $k-1$ cops, the robber can select $y^s\in[k]$ such that no cop can reach any vertex of the form $(x,y^s,z)$ where $x\in[3k]$ and $z\in\{5k+1,5k+2\}$ in less than $5k$ turns.
		The robber then chooses $z^s=5k+1$, giving a target vertex $(x^s, y^s, z^s)$.
		
		We now show that there is a path to $(x^s, y^s, z^s)$ that the robber can follow without being caught.
		Consider the subgraph induced on vertices of the form $(x^r,y,z)$ where $y\in[k]$ and $z\in\{5k+1,5k+2\}$: this graph forms a cycle, and as the cop location restrictions are satisfied, no cop can reach any such vertex in $k$ or fewer turns.
		The robber moves along a shortest path around this cycle from $(x^r,y^r,z^r)$ to $(x^r,y^s,z^s)$.
		This takes at most $k$ turns, as the cycle has $2k$ vertices and the robber can move in either direction, and so the robber cannot be caught whilst doing so.
		Next, we consider the subgraph induced on vertices of the form $(x,y^s,z^s)$ where $x\in[3k]$; this subgraph forms a path the robber will follow.
		By the choice of $y^s$, no cop could reach any vertex of this path in less than $5k$ turns before the robber started moving, and the robber has taken at most $k$ turns so far,
		so the robber can safely traverse this path to $(x^s,y^s,z^s)$, and this takes at most $3k$ turns.
		Thus traversing this path takes the robber at most $4k$ turns.
		
		Lastly, by our choice of $x^s$, the fact that the cop location restrictions held before the robber started moving, and as the robber takes at most $4k$ turns to move from $(x^r,y^r,z^r)$ to $(x^s,y^s,z^s)$,
		no cop can reach any vertex that would be within $k$ moves of any vertex of the form $(x^s, y,z)$ where $y\in[k]$ and $z\in\{5k+1,5k+2\}$ before the robber reaches $(x^s,y^s,z^s)$, so the cop location restrictions are again satisfied.
		The robber then repeats this process indefinitely to win the game.
	\end{proof}
	We can now prove Theorem~\ref{thm:no-upper-bound-by-cop-number-of-layers}.
	\begin{proof}[Proof of Theorem~\ref{thm:no-upper-bound-by-cop-number-of-layers}]
		By construction each layer is connected, and the rest of the theorem follows from Claims~\ref{claim:unbound-cop-layer-bounded-cop}, \ref{claim:unbound-robber-layer-bounded-cop}, and \ref{claim:unbound-multi-layer-cop-number}.
	\end{proof}
	
	\section{Complexity Results}\label{sec:hardness}
	
	In this section we will examine multi-layer cops and robbers from a computational complexity viewpoint.
	For a background on computational complexity, we point the reader to~\cite{sipser2012introduction}.
	First note that as determining the cop-number of a simple graph is \EXPTIME-complete~\cite{KINNERSLEY2015201}, \mlcrany is also \EXPTIME-complete by the obvious reduction that creates a multi-layer graph with one layer from a simple graph.
	The same reduction, and the fact that, unless the strong exponential time hypothesis fails\footnote{See \cite[Chapter 14]{cygan2015parameterized} for background on the strong exponential time hypothesis.}, determining if a graph is $k$-copwin requires $\Omega(n^{k-o(1)})$ time~\cite{Brandt18LowerBound}, we also get that \mlcrany requires $\Omega(n^{k-o(1)})$ time.
	
	This section begins with an $\mathcal{O}(k^2n^{2k+2})$ algorithm for solving \mlcralloc, adapted from the work of Petr, Portier, and Versteegen~\cite{Petr22FastAlgo}.
	We then show that \mlcrany is \NP-hard if each layer is a tree,
	and that \mlcr is solvable in time $\mathcal{O}(f(k,\layers)\cdot \poly(n))$ for a computable function $f$ if the robber layer is a tree.
	
	\subsection{A ``Fast'' Algorithm for Multi-Layer Cops and Robbers}
	
	An algorithm that determines whether a simple graph $G$ is $k$-copwin in $\mathcal{O}(kn^{k+2})$ time is given in~\cite{Petr22FastAlgo}.
	Petr, Portier, and Versteegen show this by first constructing a \textit{state graph} --- a directed graph $H$ wherein each vertex of $H$ corresponds to a state of a game of cops and robbers played on the original graph $G$.
	They then give an $\mathcal{O}(kn^{k+2})$ algorithm for finding all cop-win vertices of $H$, where a vertex is cop-win if the corresponding state either is a winning state for the cops, or can only lead to a winning state for the cops.
	
	We slightly adapt their techniques to give a similar result for \mlcrany.
	In particular, we will describe the construction of a directed graph $H'$ from a multi-layer graph $\mlg$ such that each vertex of $H'$ corresponds to a state of the game of multi-layer cops and robbers as played on $\mlg$.
	Once constructed, the algorithm from~\cite{Petr22FastAlgo} will identify cop-win vertices of $H'$ in $\mathcal{O}(kn^{k+2})$ time.
	
	We define the vertex set of $H'$ to be the set $[n]^{k+1} \times \mathbb{Z}_{k+1}$ such that vertex $(p_0,p_1,\ldots,p_k, t)\in [n]^{k+1} \times \mathbb{Z}_{k+1}$ corresponds to the state of the multi-layer cops and robbers game where the robber is on vertex $p_0$, cop $i$ is on $p_i$ for $i \in [n]$, and the next agent to move is the robber if $t = k$, and cop $(t+1)$ if $t \neq k$.
	
	Next, there is an edge in $H'$ between a pair of states $(q,s) = ((p_0, p_1,\ldots, p_k,t),(p'_0, p'_1,\ldots, p'_k,t'))$ if and only if:
	\begin{enumerate}
		\item $t' = (t+1) \mod (k+1)$, and
		\item either $p_t = p'_t$ or $(p_t, p'_t)$ is an edge of the layer that cop $t$ belongs to, and
		\item $p_i = p'_i$ for all $i\neq t$.
	\end{enumerate}
	
	Under this construction, it is assumed that the cops are moved one at a time, in order, such that each directed edge leads from a vertex with $t=i$ for some $i$, to a vertex with $t=i+1 \mod (k+1)$.
	This is equivalent to moving all cops simultaneously, as the cop player cannot be interrupted while moving the cops.
	
	In their paper, Petr, Portier, and Versteegen~\cite[Algorithm 1]{Petr22FastAlgo} give an $\mathcal{O}(kn^{k+2})$ algorithm for identifying all cop-win states in a state graph.
	\begin{lemma}[Rephrased from Lemma 2.2 and Theorem 2.3 in~\cite{Petr22FastAlgo}]\label{lem:fastalgoworks}
		All cop-win vertices in a state graph can be identified in $\mathcal{O}(kn^{k+2})$ time.
	\end{lemma}
	
	We now prove our result.
	\begin{theorem}
		\mlcrany can be solved in $\mathcal{O}(k^2n^{2k+2})$.
	\end{theorem}
	\begin{proof}
		We have $\mathcal{O}(kn^{k+1})$ vertices, and thus $\mathcal{O}(k^2n^{2k+2})$ edges and so creating $H'$ takes $\mathcal{O}(k^2n^{2k+2})$ time.
		Then, by construction each vertex of $H'$ represents a state of the game of \mlcrany and each arc of $H'$ corresponds to the movement (or choice to not move) of either the robber, or one cop.
		Thus, by Lemma~\ref{lem:fastalgoworks} we obtain our result.
	\end{proof}

	\subsection{\texorpdfstring{$\NP$}{NP}-hardness When Each Layer is a Tree and Some Isolated Vertices}
	The classical (single layer) cops and robbers game on a tree is solvable in polynomial time, since trees are copwin and the strategy is just to reduce distance to the robber. We show that this is not the case for the \mlcrany problem.
	\begin{theorem}\label{thm:np-by-ds}
		\mlcrany is $\NP$-hard in the number of vertices, even if each layer contains one tree and a set of isolated vertices.
	\end{theorem}
	\begin{proof}
		We show this by a reduction from dominating set which is \NP-hard~\cite{Haynes1998}.
		Let $G=(V,E)$ be an instance of dominating set, and let $n=|V|$, and for $u\in V(G)$ we let $N(u)\subseteq V(G)$ be the neighbourhood of $u$.
		Label the vertices of $G$ with elements of $[n]$, then
		create a collection of edge sets/layers as follows:
		for each $u\in V(G)$, create a layer $E_u := \{uw \mid w \in N(u)\}$.
		This can be done in polynomial time.
		
		We now show that $G$ has a dominating set of size $k$ if and only $((V, \{ E_u\}_{u\in V(G)}),k)$ is a yes-instance for \mlcrany. 
		First, assume that $G$ has a dominating set of size $k$, and let $U_d$ be such a dominating set. The cop player starts by putting a cop on vertex $u$ in layer $E_u$ for each $u \in  U_d$.
		By definition of a dominating set, for each $v \in V(G)$ there is some $u\in U_d$ such that $uv \in E(G)$.
		By the construction, this edge $uv$ is also in $E_u$, meaning that for each $v\in V(G)$ there is some $u \in U_d$ such that a cop is on $u$ in layer $E_u$ and therefore no matter where the robber starts, they will either immediately be caught or be caught after one cop move.
		
		Now assume that $G$ does not have a dominating set of size $k$, and consider an arbitrary
		allocation of $k$ cops to vertices in layers.
		Define the function $\phi$ as follows:
		For each cop $c$, if $c$ is on an isolated vertex $v$ in layer $E_i$, let $\phi(c) = v$.
		Else, $c$ must be on either the centre $v_c$ of a star in layer $E_i$, or on one of the leaves $v_\ell$ of said star, but in either case let $\phi(c) = v_c$.
		Then let $D = \{v \mid \phi(c) = v \text{ for some cop } c\}$, noting that $|D| \leq k$,
		and let $V_R = \{ u \mid u \text{ is reachable by some cop }\}$.
		Note that $V_R$ contains exactly the vertices in, or adjacent to, any vertex in $D$.
		Thus, if $V_R = V$ then $D$ is a dominating set of size $k$.
		However, as $G$ does not have a dominating set of size $k$, the set $V\setminus V_R$ must be non-empty.
		The robber can then start on any vertex in $V\setminus V_R$ and stay there forever to win.
	\end{proof}
	
	The next theorem gives a polynomial-time algorithm for determining if $k$ cops can win on a given multi-layer graph if the robber layer is a tree.
	In particular, this result applies even if the layers are not connected.
	\begin{theorem}\label{thm:mlcr-trees}
		Given a multi-layer graph $\mlg = (V, \{C_1,\ldots,C_\layers\}, R)$,
		if $R$ is a tree, then \mlcr on $\mlg$ can be solved in time $\mathcal{O}(f(k,\layers)\cdot \poly(n))$, where $k$ is the number of cops, $\layers$ is the number of layers of $\mlg$, $f$ is a computable function independent of $n$, and $\poly(n)$ is a fixed polynomial in $n$.
	\end{theorem}
	\begin{proof}
		If, given an allocation of cops to layers, we can answer \mlcralloc in time $\mathcal{O}(f(k,\layers)\cdot \poly(n))$, then as there are
		$\layers^k$ ways to choose a layer for each cop, we can also solve \mlcr in time $\mathcal{O}(f(k,\layers)\cdot \poly(n))$.
		Note that this allows scenarios where multiple cops are on the same layer.
		Thus, for the rest of this proof we assume that we know how many cops are assigned to each layer.
		If there is an edge $uv$ in the robber layer such that both $u$ and $v$ are only reachable by at most one cop, say $c_i$, and if the distance from $u$ to $v$ for $c_i$ is at least $3$, the robber player wins by starting on, and staying on, the vertex from $\{u,v\}$ that is farthest from $c_i$.
		We will call such an edge a \emph{robber's edge}, so that if a robber's edge exists then the robber can win.
		We next show that if no robber's edge exists, the cops can win, and
		then show that determining if a robber's edge exists can be done efficiently.
		
		Assume that there is no robber's edge.
		The cop player can then win by iterating the following process.
		Note that each iteration will consist of zero or more rounds of the actual game.
		Let $S_0$ be the set of vertices that are reachable by the robber without going through a vertex occupied by a cop after the cops and the robber have been placed on the graph, but before any moves take place. Similarly, let $S_t$ be the set of vertices that the robber can reach without going through a vertex occupied by a cop after $t$ iterations of the cop strategy, which we now outline.
		
		As we are giving the cop strategy, we assume that it is now the cop player's turn.
		Let $u$ be the location of a cop that is closest to the robber in the robber layer.
		If there are multiple such locations, choose one arbitrarily.
		Then let $w$ be the vertex that is adjacent to $u$ in the robber layer and is strictly closer to the robber's current position than $u$. Such a $w$ must exist as $R$ is a tree.
		The cop player moves a cop to $w$ such that the cop on $u$ either stays on $u$, moves directly to $w$, or moves from $u$ to some vertex $u'$ and then (unless the robber has now moved to $u$, in which case this cop captures the robber by moving back to $u$) moves to $w$.
		This is possible as there is no robber's edge.
		Then $S_1 = S_0 \setminus \{u\}$, as the only way the robber can leave $S_0$ is via $v$, which is occupied or threatened by a cop until a cop reaches $u$, but as $R$ is a tree, $u$ is the second-last vertex on any path the robber can take to $v$.
		More generally, $|S_{t+1}| \leq |S_t| - 1$.
		Thus, after $|S_0|$ iterations the robber has zero vertices which they can occupy and must be caught.
		
		Thus determining if the cop player can win turns into determining, for a given allocation, if there exists a robber's edge.
		This can be done by inspecting each edge in the robber layer and determining how many cops can reach either endpoint.
		If two or more cops can reach either endpoint, we are done with this edge, else only one cop can reach either end point, and we need only check distance in the layer occupied by this cop.
		This can all be checked in time polynomial in $n$ --- for each vertex we can determine the number of cops that can reach said vertex by a simple graph traversal algorithm.
	\end{proof}
	
	As a quick corollary, if even one cop layer is connected, and we have at least two cops, by putting both cops on this layer there is no edge $uv$ such that its two endpoints $u$ and $v$ are only reachable by one cop, giving the following.
	
	\begin{corollary}
		Given a multi-layer graph $\mlg = (V, \{C_1,\ldots,C_\layers\}, R)$,
		if $R$ is a tree, and at least one cop layer is connected, then $\mcop{\mlg} \leq 2$.
	\end{corollary}

	\section{Extremal Multi-Layer Cop-Number} \label{sec:upper-bounds} 
	In this section we study, for a given  connected graph $G = (V,E)$, the extremal multi-layer cop number of $G$. This is the multi-layer cop number maximised over the set of all multi-layer graphs with connected cop-layers, which when flattened give $G$. More formally, for  given connected graph $G=(V,E)$, if we define the set
	\[\mathcal{L}(G,\layers)  = \{(V, \{C_1, \dots, C_\layers\}, *) : E = C_1 \cup \cdots \cup C_\layers \text{ and for each }i\in[\layers], (V,C_i) \text{ is connected}\} , \] then the \text{extremal multi-layer cop-number} of the (single-layer) graph $G$ is given by 
	\begin{equation*}\label{eq:extremalCN}
		\maxcop{\layers}{G} = \max_{\mlg\in \mathcal{L}(G,\layers)} \mcop{\mlg}.
	\end{equation*}
	
	Observe that we set $R=E$ as any choice of robber layer $R\subsetneq E$ gives a cop number at most that with $R=E$ by Proposition \ref{prop:subsetrobber}. \begin{remark}\label{rem:cycleexample}We  restrict the problem to connected layers to avoid pathological examples; in particular if each layer is not connected then one can treat each component as a separate layer which can lead to a larger than expected cop number.
		For example consider a $(2n)$-vertex cycle which has two (disconnected) layers that are edge disjoint matchings on $2n$ vertices as its two cop layers. That is, \[C_1= \{(2i-1)(2i) : i \in [n]\}, \quad \text{ and }\quad C_2= \{(2n)1\}\cup \{(2i)(2i+1) : i \in [n-1]\}.\] Since any cop can only patrol at most two vertices, if there are at most $n-1$ cops then there must be at least one vertex no cop can reach, thus in this example the multi-layer cop number is $n=|V|/2$. \end{remark}
	
	We begin in Sections \ref{sec:mlexc} and \ref{sec:mlds} by generalising two tools for bounding the cop number of graphs to the setting of multi-layer graphs; $(1,k)$-existentially closed graphs and bounds by domination number. We then apply these tools to determine the extremal multi-layer cop number of cliques and binomial random graphs in Section \ref{sec:clique} and \ref{sec:RG} respectively. Finally, in Section \ref{sec:trw}, we show that the extremal multi-layer cop number of a graph is bounded from above by its treewidth.

	\subsection{Multi-Layer Generalisation of Existential Closure} \label{sec:mlexc}
	
	The ``$(1,k)$-existentially closed'' property was introduced in \cite[Page 427]{BonatoPW07} for proving lower bounds on the cop number of a random graph.
	Given a positive integer $k$, a graph $G$ is \emph{$(1,k)$-existentially closed} if for each $k$-element subset $S \subseteq V(G)$, and vertex $u\not\in S$, there is a vertex $z \not\in S$ adjacent to $u$, which is not adjacent to any vertex in $S$. This definition is useful for lower bounding the cop number as if a robber is at any vertex $u$ of a $(1,k)$-existentially closed graph, then in the next step they can always move to a vertex $z$ that is safe for at least one more step. We adapt this to the multi-layer setting and use it to obtain lower bounds for the complete graph and the binomial random graph.
	
	For integers $\layers, k > 0$ we say that a multi-layer graph $(V,\{C_1, \dots, C_\layers\}, R)$ is $(1,k)$-multi-layer existentially closed (or $(1, k)$-m.e.c.)
	if for any subsets $S_{1}, \dots, S_{\layers} \subset V$, where $\sum_{i=1}^\layers |S_{i}| =  k$, we have $S_1\cup \cdots \cup S_\layers \neq V$, and for any $v \notin S_{1}\cup \cdots\cup S_{\layers}$
	there exists $x\in V\backslash (\{v\}\cup S_{1}\cup \cdots\cup S_{\layers})$ such that $vx \in R$ and $s_{i}x\not\in C_{i}$ for any $s_{i} \in S_{i}$ and $i\in [\layers]$. 
	
	As is the case for classical cops and robbers, we get a lower bound on the multi-layer cop number almost immediately from the definition.

	\begin{lemma} \label{lem:mec}
		If a multi-layer graph $\mlg$ is  $(1, k)$-m.e.c., then $\mcop{\mlg}> k$. 
	\end{lemma}
	\begin{proof}
		If the robber is at any vertex $v$ at any time before capture then, since $\mlg$ is $(1, k)$-m.e.c., for any placement of $k$ cops in $V\backslash \{v\}$ onto any layers (where the set of layer $i$ cops is $S_i$), there always exists a vertex $x\notin \{v\}\cup S_1\cup \cdots \cup S_\layers$ adjacent to $v$ that is not adjacent to any cop. Thus, in the next step the robber at $v$ can move to $x$ and since this $x$ is not adjacent to any cop in any layer the robber is safe from all cops in that step. Continuing inductively the robber can avoid capture indefinitely. 
	\end{proof}
	
	We use the above definition and lemma to give a lower bound on the multi-layer cop number when the robber layer contains all edges (e.g.,~the flattened graph is the complete graph $K_n$).
	For this, we define $N_{C_i}[v]= \{w \mid vw \in C_i\}\cup\{v\}$ to be the closed neighbourhood of $v$ in $C_i$.
	
	\begin{lemma}\label{lem:1mec}Let $\layers,k>0$ be integers, $\mlg=(V,\{C_1, \ldots, C_\layers\}, \binom{V}{2})$, and $k<|V|$. If for any  $S_{1}, \dots, S_{\layers} \subset V $  with $\sum_{i=1}^\layers |S_{i}|= k$ and $u\notin  \cup_{i\in[\layers]}S_i$, we have $\left|\{u\}\cup \bigcup_{i=1}^{\layers}\bigcup_{v\in S_{i}}N_{S_i}[v] \right| < n $, then $\mcop{\mlg}> k$.
	\end{lemma} 
	\begin{proof}
		Since for any choice of the cop positions the set of vertices either occupied by the robber, occupied by a cop, or adjacent to a cop in some layer is at most $n-1$, there must be some vertex $x\in V$ for which none of these things hold. Thus, as the robber has all edges available (i.e.\ $R=\binom{V}{2}$), we see that $\mlg$ is $(1, k)$-m.e.c., and thus the result holds by Lemma \ref{lem:mec}.
	\end{proof}
	\subsection{Multi-layer Dominating Set}\label{sec:mlds}
	Let $\mlg = (V, \{E_1, \ldots, E_\layers\})$ be a multi-layer graph (without designated cop/robber layers). A \textit{multi-layer dominating set} in $\mlg$ is a set $D\subseteq V \times \{1,\ldots,\layers\}$ of vertex-layer pairs such that for every $v\in V$, either $(v,i)\in D$ for some $i$, or there is a $(w,i)\in D$ such that $w\in V$ and $vw\in E_i$. We define the \textit{domination number} $\gamma(\mlg)$ of $\mlg$ to be the size of a smallest multi-layer dominating set in $\mlg$. Note that if $\mlg$ has a single-layer this definition aligns with the traditional notion of dominating set, which justifies the overloaded notation. It is a folklore result that the cop number is at most the size of any dominating set in the graph, this also holds in the multi-layer setting.  
	
	\begin{lemma}\label{lem:copdom}
		Let $\mlg:=(V, \{E_1, \dots, E_\layers\})$ be any multi-layer graph and $\mlg':=(V,  \{E_1, \dots, E_\layers\},\binom{V}{2})$. Then, $\mcop{\mlg'} \leq \gamma(\mlg) $.  
	\end{lemma}
	
	\begin{proof}Let $D$ be any multi-layer dominating set of size $|D|=\gamma(\mlg)$, and for each $(v,i)\in D$ place one cop in layer $i$ at the vertex $v$. The result now follows as if the robber is at an any vertex then they are adjacent to a cop in some layer and so the robber will be caught after the cops first move. \end{proof}

	We now introduce the parameter $\delta(\mlg)$ which is an analogue of minimum degree for a multi-layer graph $\mlg =(V,\{E_1,\dots, E_\layers\})$. This is given by  \begin{equation}\label{eq:deltamlg}\delta(\mlg) := \min_{v\in V} \sum_{i\in[\layers]} d_{(V,E_i)}(v).\end{equation}
	
	Using this notion we prove a bound on the domination number of a multi-layer graph based on a classic application of the probabilistic method \cite[Theorem 1.2.2]{AlonSpencer}. 	\begin{proposition}\label{lem:domnum}
		Let $\mlg=(V, \{E_1, \dots, E_\layers\})$ be any multi-layer graph. 
		Then, \[\gamma(\mlg)\leq   \frac{n\layers}{\layers + \delta(\mlg)}\cdot  \left(  \ln\left(\frac{\layers + \delta(\mlg)}{\layers}\right)+ 1 \right) .\]
	\end{proposition}
	\begin{proof}We select a random subset $D\subseteq  V\times[\layers]$ by including each  element $(v,i)\in V\times[\layers]$ with probability $p$. Let $Y= V\backslash \bigcup_{(v,i)\in D} N_{(V,E_i)}[v]$ be the set of vertices not dominated by $D$. Thus if $Y'\subseteq  V\times[\layers]$ is a set formed by taking each element of $y\in Y$ and assigning it to an arbitrary layer then $D\cup Y'$ is a multi-layer dominating set. Recall that $1 -p \leq  e^{-p}$ for $p\geq 0$, and observe that each vertex $v\in V$ is in $Y$ with probability at most 
		\[(1-p)^{\layers}\cdot (1-p)^{d_{(V,E_1)}(v)}\cdots(1-p)^{d_{(V,E_\layers)}(v)} = (1-p)^{\layers + \sum_{i\in[\layers]} d_{(V,E_i)}(v)}\leq (1-p)^{\layers + \delta(\mathcal{G})} \leq e^{-p\cdot(\layers + \delta(\mathcal{G}) )}.  \]Thus by linearity of expectation and disjointness of $D$ and $Y'$ we have \begin{equation}\label{eq:Ebdd}\mathbb{E}[|D\cup Y'|] = \mathbb{E}[|D|]+ \mathbb{E}[|Y'|]\leq  n\layers p + ne^{-p\cdot(\layers + \delta(\mathcal{G}) )}.\end{equation} 
		If we set $ p = \ln\left(\frac{\layers + \delta(\mathcal{G})}{\layers}\right)/ (\layers + \delta(\mathcal{G}))$ then \eqref{eq:Ebdd} gives \[\mathbb{E}[|D\cup Y'|] \leq  n \layers \cdot \frac{\ln\left(\frac{\layers +\delta(\mathcal{G})}{\layers}\right)}{\layers +\delta(\mathcal{G})} + n\cdot e^{ - \ln\left(\frac{\layers +\delta(\mathcal{G})}{\layers}\right) } =\frac{n\layers}{1 + \delta(\mathcal{G})}\cdot  \left(  \ln\left(\frac{\layers + \delta(\mathcal{G})}{\layers}\right)+ 1 \right)   . \]  So there must exist a multi-layer dominating set in $\mlg$ with at most this many vertices, giving the result. 
	\end{proof}
	
	Note that there are at least two other sensible definitions of minimum degree for multi-layer graphs, namely the minimum degree of each layer, i.e.\ $\min_{i\in [\layers]}\min_{v\in V}d_{(V,E_i)}(v)$, and minimum number of neighbours within any layer, given by $\delta(\flatten{\mlg})$. Our definition of $\delta(\mlg)$ above in \eqref{eq:deltamlg} can be thought of as the ``minimum number of edges incident in any layer'', this is arguably a less natural notion than $\delta(\flatten{\mlg})$ however it gives a better bound in our application (Proposition \ref{lem:domnum}), the following simple result gives an indication of why. 
	
	\begin{proposition}\label{prop:multmin>flatmin} For any multi-layer graph $\mlg$ we have $\delta(\flatten{\mlg})\leq \delta(\mlg) $. 
	\end{proposition}\begin{proof}
		Observe that any edge in $\flatten{\mlg}$ may appear in more than one distinct layer of $\mlg$. Thus, for any $v\in V$, 
		\begin{equation*} d_{\mathsf{fl}(\mlg)}(v) \leq \sum_{i\in [\layers]}d_{(V,E_i)}(v).\end{equation*}Since this holds for all $v\in V$, the result follows by taking the minimum over all $v\in V$ on both sides.  \end{proof}

	\subsection{Complete Graph}\label{sec:clique}
	The aim of this section is to prove the following result. \begin{theorem}\label{thm:complete}
		Let $n\geq 1$ and $1\leq \layers< \lfloor \frac{n}{2}\rfloor $ be integers. Then, $\lceil \frac{\layers}{10}\rceil \leq \maxcop{\layers}{K_{n}}\leq \layers $.     
	\end{theorem} 

	It is the lower bound of Theorem \ref{thm:complete} which is the non-trivial direction. The idea is to partition the edges of the complete graph into almost-regular graphs with no overlapping edges. If we take each of these graphs as a cop layer then we can show that the resulting multi-layer graphs is $(1,k)$-multi-layer existentially closed, the result then follows from Lemma \ref{lem:1mec}. An illustration of the edge partition is given in Figure \ref{fig:cliqueconstruction}.

	\begin{proof}[Proof of Theorem \ref{thm:complete}]For the upper bound observe that we can choose any vertex $v$ of the clique and place one cop within each of the $\layers$ layers at the vertex $v$. Since any other vertex $u$ is adjacent to $v$ in at least one layer we can guard every vertex of $K_n$. 
		
		The proof of the lower bound will be based on building cop layers using an elegant $(2n-1)$-colouring of $K_{2n}$ from Sofier's book \cite[Problem 16.4]{Sofier}, merging colour classes, and adding a single extra vertex in the case of odd $n$. We first state the colouring from \cite[Problem 16.4]{Sofier}, which Sofier in turn attributes to \cite{BCL}:

		\begin{idea}\label{eq:SofierColouring}
			\textit{($2\ell-1$)-Colouring of $K_{2\ell}$:} Let $V(K_{2\ell })= \{0,1,\dots, 2\ell-1\}$. Arrange the vertices $0 ,\dots , 2\ell-2$ as the vertices of a regular
			$(2\ell - 1)$-gon, and place $2\ell- 1$ in its center.
			For each $i\in \{0, \ldots, 2\ell-2\}$, assign colour $c_i$ to the edge $(2\ell-1)(i)$ and also assign colour $c_i$ to all edges that are perpendicular to $(2\ell-1)(i)$ (i.e., for $j\in[\ell-1]$ assign colour $c_i$ to all edges of the form $(i-j \bmod 2\ell-1)(i+j \bmod 2\ell-1)$).
		\end{idea}
		Using this we now describe how to partition $K_n$ in to $1\leq \layers<\left\lfloor \frac{n}{2} \right\rfloor$ layers, for odd or even $n$.  
		\begin{idea}\label{eq:Partition}
			\textit{Partition of $K_n$ into $\layers$ layers:} For $n=2\ell$ take the ($2\ell-1$)-colouring of $K_{2\ell}$ given by \eqref{eq:SofierColouring}. Then, partition the integer interval $\mathcal{I}=\{1,\dots,2\ell -1 \}$ into $\layers$ disjoint sub-intervals $\mathcal{I}_1, \dots, \mathcal{I}_\layers $ of consecutive integers, each containing either $\lfloor\frac{2\ell -1}{\layers} \rfloor $ or $\lceil\frac{2\ell -1}{\layers} \rceil $ integers. Now, for each $i\in [\layers]$, identify each of the colour classes $c_j$, where $j\in \mathcal{I}_i$, as a single cop layer $C_i$.
			For the case when $n = 2\ell+1$ is odd we first partition $K_{2\ell}$ into $\layers$ layers as in the even case, and then add vertex $2\ell$ as well as edges $(2\ell)(i)$ for $i\in[2\ell-1]$ to the $\layers$ layers in any way such that vertex $2\ell$ is incident to either $\lfloor\frac{2\ell}{\layers} \rfloor $ or $\lceil\frac{2\ell}{\layers} \rceil $ edges in each layer.
		\end{idea}

		See Figure \ref{fig:cliqueconstruction} for an illustration of the construction. We now prove the following claim. 
		\begin{claim}\label{clm:partition}
			For any $n\geq 1$ and $1\leq \layers< \lfloor n/2\rfloor $ the partition \eqref{eq:Partition} gives a set of connected $n$-vertex graphs $\{G_i=(V, C_i)\}_{i\in [\layers]}$ satisfying $K_n=G_1\cup  \dots \cup  G_\layers$ and $\max_{i\in [\layers]}\max_{v\in V}d_{G_i}(v)\leq \lceil  {n/\layers} \rceil$.   
		\end{claim}
		
		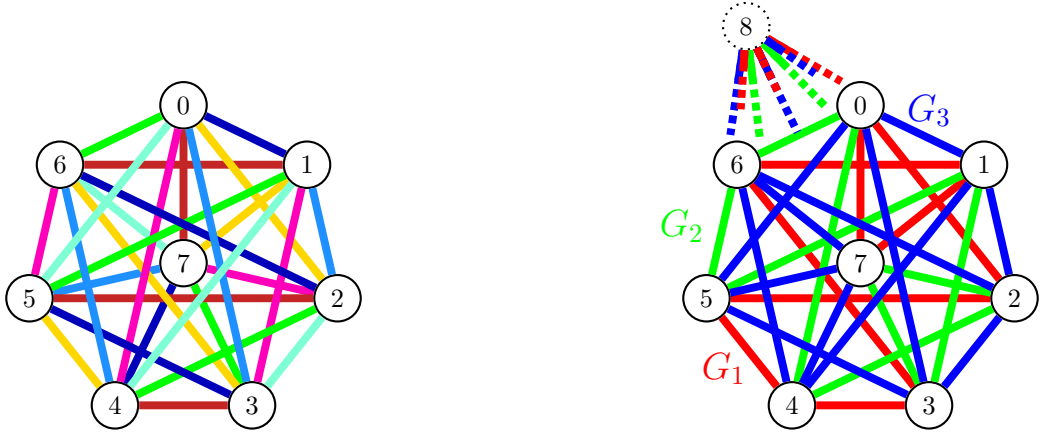
\begin{figure}
			\centering
			\begin{tikzpicture}[node/.style={thick,circle,draw=black,minimum size=.4cm,fill=white},dotnode/.style={thick,dotted,circle,draw=black,minimum size=.4cm,fill=white},edge0/.style={line width=3pt,col0}, edge1/.style={line width=3pt,col1},edge2/.style={line width=3pt,col2},edge3/.style={line width=3pt,col3},edge4/.style={line width=3pt,col4},edge5/.style={line width=3pt,col5},edge6/.style={line width=3pt,col6},edgea/.style={line width=3pt,cola},edgeb/.style={line width=3pt,colb},edgec/.style={line width=3pt,colc}]
				
				\begin{scope}[shift={(-4.5,0)},scale=.7]	\node[node] (7) at (0,0) {$7$};
					\def \radius {3cm}
					\foreach \s in {0,...,6}
					{
						\node[node] (\s) at ({360/7 * (- \s +  3 ) - 900/14 }:\radius)  {$\s$};;
					}
					
					\foreach \s in {0,...,6}
					{
						\def\a{\s +1}
						\draw[edge\s] ( \s    ) to (7);
					}
					
					\draw[edge0] (1) to (6);
					\draw[edge0] (2) to (5);
					\draw[edge0] (3) to (4);
					
					\draw[edge1] (0) to (2);
					\draw[edge1] (6) to (3);
					\draw[edge1] (5) to (4);
					
					\draw[edge2] (1) to (3);
					\draw[edge2] (0) to (4);
					\draw[edge2] (6) to (5);
					
					\draw[edge3] (4) to (2);
					\draw[edge3] (5) to (1);
					\draw[edge3] (6) to (0);	
					
					\draw[edge4] (3) to (5);
					\draw[edge4] (2) to (6);
					\draw[edge4] (1) to (0);	
					
					\draw[edge5] (4) to (6);
					\draw[edge5] (3) to (0);
					\draw[edge5] (2) to (1);
					
					\draw[edge6] (5) to (0);
					\draw[edge6] (4) to (1);
					\draw[edge6] (3) to (2);	
				\end{scope}
				
				\begin{scope}[shift={(4.5,0)},scale=.7]	\node[node] (7) at (0,0) {$7$};
					\def \radius {3cm}
					\foreach \s in {0,...,6}
					{
						\node[node] (\s) at ({360/7 * (- \s +  3 ) - 900/14 }:\radius)  {$\s$};;
					}
					
					\draw[edgea] (7) to (0);
					\draw[edgea] (1) to (6);
					\draw[edgea] (2) to (5);
					\draw[edgea] (3) to (4);
					
					\draw[edgea] (7) to (1);
					\draw[edgea] (0) to (2);
					\draw[edgea] (6) to (3);
					\draw[edgea] (5) to (4);
					
					\draw[edgeb] (7) to (2);
					\draw[edgeb] (1) to (3);
					\draw[edgeb] (0) to (4);
					\draw[edgeb] (6) to (5);
					
					\draw[edgeb] (7) to (3);
					\draw[edgeb] (4) to (2);
					\draw[edgeb] (5) to (1);
					\draw[edgeb] (6) to (0);	
					
					\draw[edgec] (7) to (4);
					\draw[edgec] (3) to (5);
					\draw[edgec] (2) to (6);
					\draw[edgec] (1) to (0);	
					
					\draw[edgec] (7) to (5);
					\draw[edgec] (4) to (6);
					\draw[edgec] (3) to (0);
					\draw[edgec] (2) to (1);

					\draw[edgec] (7) to (6);
					\draw[edgec] (5) to (0);
					\draw[edgec] (4) to (1);
					\draw[edgec] (3) to (2);

					\draw (-2,-2) node[anchor=east]{{\color{cola}\Large $G_1$}};
					\draw (-2.8,.7) node[anchor=east]{{\color{colb}\Large $G_2$}};
					\draw (1.9,2.9) node[anchor=east]{{\color{colc}\Large $G_3$}};
					\node[dotnode] (8) at ({360/7 * (3.5 ) - 900/14 }:5cm)  {$8$};
					
					\draw[] (8) edge [line width=3pt,colc] ($(8)!0.2!(5)$) edge [dotted,line width=3pt,colc] ($(8)!0.40!(5)$);
					\draw[] (8) edge [line width=3pt,cola] ($(8)!0.2!(1)$) edge [dotted,line width=3pt,cola] ($(8)!0.40!(1)$);
					\draw[] (8) edge [line width=3pt,colb] ($(8)!0.15!(4)$) edge [dotted,line width=3pt,colb] ($(8)!0.30!(4)$);	
					\draw[] (8) edge [line width=3pt,colc] ($(8)!0.14!(3)$) edge [dotted,line width=3pt,colc] ($(8)!0.28!(3)$);
					\draw[] (8) edge [line width=3pt,cola] ($(8)!0.14!(7)$) edge [dotted,line width=3pt,cola] ($(8)!0.28!(7)$);
					\draw[] (8) edge [line width=3pt,colb] ($(8)!0.15!(2)$) edge [dotted,line width=3pt,colb] ($(8)!0.30!(2)$);
					\draw[] (8) edge [line width=3pt,cola] ($(8)!0.3!(6)$) edge [dotted,line width=3pt,cola] ($(8)!0.60!(6)$);
					\draw[] (8) edge [line width=3pt,colc] ($(8)!0.3!(0)$) edge [dotted,line width=3pt,colc] ($(8)!0.60!(0)$);	
				\end{scope}

			\end{tikzpicture}
			\caption{On the left is a copy of $K_{8}$ coloured according to \eqref{eq:SofierColouring}, on the right is the partition of $K_8$ (or $K_9$ if the vertex $v_8$ is included) given by \eqref{eq:Partition}.}\label{fig:cliqueconstruction}
		\end{figure}
		\begin{poc}For the even case $n=2\ell$ it follows from \cite[Problem 16.4]{Sofier} that the colouring given by \eqref{eq:SofierColouring} is a valid ($2\ell-1$)-colouring of $K_{2\ell}$ and so $C_1 \cup \cdots \cup C_\layers = E(K_{2\ell})$.
			We claim that if any two colour classes $c_{i}$ and $c_{i+1}$ from this ($2\ell -1$)-colouring are merged then the resulting set of edges contains a Hamiltonian cycle. To see this holds note that class $c_{i}$ contains the edges $(2\ell -1)(i)$ and  $\{(i+j  \text{ mod }  2\ell-1)(i-j  \text{ mod }  2\ell-1)\}_{j\in [\ell -1]}$ and class $c_{i+1}$ contains $(2\ell -1)(i+1) $ and  $\{(i+1-j \text{ mod } 2\ell-1)(i+1+j  \text{ mod } 2\ell-1)\}_{j\in [\ell -1]}$.
			Thus in the union of classes $c_i$ and $c_{i+1}$ there is a Hamiltonian cycle given by 
			
			\begin{equation*}\begin{aligned}& 2\ell -1 ,\; i,\; i+2 \text{ mod } 2\ell-1,\;  i-2  \text{ mod } 2\ell-1 ,\;i+4  \text{ mod } 2\ell-1 ,\;i-4  \text{ mod } 2\ell-1 ,\dots \\ 
					&\qquad \qquad \dots,\;i+3  \text{ mod } 2\ell-1,\; i-1  \text{ mod } 2\ell-1 ,\; i+1,\;2\ell-1,  \end{aligned}\end{equation*}
			 where edge colours alternate between $c_i$ and $c_{i+1}$.
			Recall that, for $n=2\ell$, each layer is formed from at least $\lfloor \frac{2\ell-1 }{\layers} \rfloor $ consecutive colour classes, where $\layers < \lfloor \frac{n}{2}\rfloor = \ell$. Thus, each layer contains at least $\lfloor \frac{2\ell - 1}{\ell -1 } \rfloor \geq 2 $ consecutive colour classes, and so each layer is connected in $K_{n}$. For the odd case $n=2\ell+1$ we are connecting the extra vertex (labelled $2\ell$) to at least one other vertex of the even vertex construction in each layer. So in the odd case the layers are also connected.
			
			For the bound on the degrees, when $n=2\ell$ each layer of the construction contains at most $\lceil \frac{2\ell -1}{\layers} \rceil=\lceil \frac{n -1}{\layers} \rceil$ colour classes merged to form a layer, and the number of edges from any colour class adjacent to any vertex is exactly one. When $n=2\ell +1$, the ``extra'' vertex (labelled $2\ell$) has at most $\lceil \frac{n-1}{\layers} \rceil$ edges from any layer leaving it and adds at most one edge from each layer to any vertex in $\{0, \dots, 2\ell -1\}$. Thus, in the odd case, the max degree is bounded by $\max\{ \lceil \frac{n-1}{\layers} \rceil, \lceil \frac{n-2}{\layers} \rceil + 1\}\leq \lceil \frac{n }{\layers} \rceil  $ as claimed.         
		\end{poc}
		We now apply Lemma \ref{lem:1mec} to the partition found by \eqref{eq:Partition} to show the result. Recall that the degree of any vertex within any cop layer $C_i$ is at most $\left\lceil  {n/\layers}\right\rceil  \leq n/\layers+1$ by Claim \ref{clm:partition}. Thus, setting $k = \lfloor \layers/10\rfloor \leq \layers/10   $ in  Lemma \ref{lem:1mec}, for any vertex $u\in V$, and any sets $S_{1}, \dots, S_{\layers} \subset V\backslash \{v\}$  satisfying $\sum_{i=1}^\layers |S_{i}|= k$ we have  
		\begin{equation*}  \left|\{u\}\cup \bigcup_{i=1}^{\layers}\bigcup_{v\in S_{i}}  N_{G_{i}}[v] \right|\leq 1 + k + k\cdot\left(\left\lceil \frac{n}{\layers}\right\rceil +1\right)  \leq  1 +  \frac{3\layers }{10} +  \frac{n }{10}\leq \frac{n}{4} +1 <n,  \end{equation*}
		where the last inequality holds as we can assume $n\geq 2$. The result then follows from Lemma \ref{lem:1mec}. 
	\end{proof}

	\subsection{Binomial Random Graphs}\label{sec:RG}
	We now consider the extremal multi-layer cop number of the binomial random graph $G_{n, p}$. For any integer $n\geq 1$, this is the  $n$-vertex simple graph generated by sampling each possible edge independently with probability $0< p=p(n)<1$, see \cite{RandomGraphBook} for more details. The following result shows that, for a suitably dense binomial random graph $G_{n,p}$, with probability tending to $1$ as $n\rightarrow \infty$, we have  $\maxcop{\layers}{G_{n,p}} =\Theta(\layers\log(n)/p)$. The single-layer cop number of $G_{n,p}$ in the same range is known to be $ \Theta(\log(n)/p) $ \cite{BonatoPW07}, so in some sense our result provides a generalisation.

	\begin{theorem}\label{thm:randgraph}
		For any fixed $\eps > 0$, let $ n^{1/2+\eps }\leq np=o(n)$ and $1\leq \layers\leq n^{\eps}$. Then,
		\[\mathbb{P}\left(\frac{\eps }{10} \cdot \frac{ \layers \cdot \ln  n}{ p}\leq \maxcop{\layers}{G_{n,p}} \leq  10\cdot  \frac{ \layers \cdot \ln  n}{ p}\right)\geq  1-e^{-\Omega(\sqrt{n})}  .\]   
	\end{theorem}
	
	Our proof follows the broad idea of \cite{BonatoPW07} for the single layer case, namely bounding the cop number from above by the domination number and from below using the existentially closed property. However, in our lower bound there is the added complication that we need to construct a random multilayer graph where each layer is connected, and when flattened follows the distribution of $G_{n,p}$.  
	
	We first recall a standard Chernoff bound. 
	\begin{lemma}[{\cite[Corollary 4.6]{PandC}}]\label{lem:chb} Let $X_1,\dots, X_n$ be independent Bernoulli trials such that $\Pr{X_i=1} =p_i$. Let $X =\sum_{i=1}^n X_i$ and $\mu =\Ex{X}$. Then, for any $0<\delta<1$, we have 
		$\Pr{|X-\mu|\geq \delta \mu}\leq 2 e^{-\delta^2\mu /3}.$
	\end{lemma}
	
	We must secondly prove a standard lemma on the connectivity of $G_{n,p}$. 
	\begin{lemma}\label{lem:randomgraphconprob} Let $n\geq 1000$ and $ \frac{10 \log n}{n}\leq p \leq 1$. Then,  $\mathbb{P}(G_{n,p} \text{ is connected}) \geq 1 - e^{- np/3}$.
	\end{lemma}
	\begin{proof}
		We assume for the time being that $p\leq 5/6$. Observe that if $G_{n,p}$ is not connected then it must contain a connected component of size $1\leq k \leq \lfloor n/2\rfloor $ that has no edges to any other vertices of the graph. Thus, if $X_k$ is the number of connected components of size exactly $k$ in $G_{n,p}$ then, by the union bound
		\[\mathbb{P}(G_{n,p} \text{ is not connected}) \leq \mathbb{P}(X_1>0) + \sum_{k=2}^{\lfloor n/2\rfloor} \mathbb{P}(X_k>0).  \]
		Observe that a connected component of size one is a vertex with no edges, thus by the union bound \[\mathbb{P}(X_1>0) \leq n \cdot (1-p)^{n-1}  \leq n\cdot\frac{1}{1-p} \cdot e^{-np} \leq 6n\cdot e^{-np} \leq  e^{-np/2}. \]For $2\leq k\leq \lfloor n/2\rfloor $ a connected component must contain a spanning tree. By Cayley's formula there are $k^{k-2}$ possible spanning trees, each such tree occurs with probability $p^{k-1}$, and there must be no edges going outside the component, which happens with probability $(1-p)^{k(n-k)}$. As there are $\binom{n}{k}$ ways to choose the vertices of a component on $k$ vertices, by the union bound, for any $2\leq k\leq \lfloor n/2\rfloor $, we have  \[\mathbb{P}(X_k> 0) \leq \binom{n}{k}\cdot k^{k-2}p^{k-1}\cdot (1-p)^{k(n-k)} \leq n^k\cdot \frac{(pk)^k}{p}\cdot e^{-knp/2} \leq e^{-k(np/2 - \ln(npk) ) }/p \leq e^{-knp/3},  \] as $np\geq 10 \log n$. Combining these two bounds and applying the geometric series formula gives 
		\[\mathbb{P}(G_{n,p} \text{ is not connected}) \leq e^{-np/2} + \sum_{k=2}^{\lfloor n/2\rfloor}e^{-knp/3} \leq   e^{-np/2} + \frac{e^{-2np/3}}{1-e^{-np/3}}  \leq e^{-2np/5},  \] as $n\geq 100$. Finally, for $5/6<p\leq 1$, since connectivity is a monotone property  we have \[\mathbb{P}(G_{n,p} \text{ is not connected}) \leq \mathbb{P}(G_{n,5/6} \text{ is not connected})\leq e^{-2n\cdot (5/6)/5} = e^{-n/3}\leq e^{-np/3} ,\] so the result follows.     
	\end{proof}
	We can now prove  Theorem \ref{thm:randgraph}. 
	\begin{proof}[Proof of Theorem \ref{thm:randgraph}] 
		We begin with the upper bound. This is established by bounding the multi-layer domination number $\gamma(\mathcal{G})$  from above, where $\mathcal{G}= (V, \{E_1, \ldots, E_\layers\})$ is constructed by 
		sampling a random graph $G_{n,p}$ with vertex set $V$ and choosing an arbitrary partition $\{E_i\}_{i\in [\layers]}$ of its edge set $E$ into layers satisfying $\cup_{i\in \layers} E_i =E$. Recall the definition $\delta(\mlg) := \min_{v\in V} \sum_{i\in[\layers]} d_{(V,E_i)}(v)$ from \eqref{eq:deltamlg} and observe that $\delta(\mlg) \geq \delta(\flatten{\mlg})= \delta(G_{n,p})$ holds for any partition $\{E_i\}_{i\in [\layers]}$ by Proposition \ref{prop:multmin>flatmin}. Thus, Proposition \ref{lem:domnum} gives \begin{equation}\label{eq:dombydeg}\gamma(\mlg)\leq   \frac{n\layers}{\layers + \delta(\mlg)}\cdot  \left(  \ln\left(\frac{\layers + \delta(\mlg)}{\layers}\right)+ 1 \right)\leq    \frac{n\layers}{\layers + \delta(G_{n,p})}\cdot  \left(  \ln\left(\frac{\layers + \delta(G_{n,p})}{\layers}\right)+ 1 \right) , \end{equation} where the second inequality holds since $\frac{\layers}{\layers + x}\cdot   \ln\left(\frac{\layers + x}{\layers}\right) $ is decreasing in $x$ whenever $x \geq \layers (e-1) $.
		
		Observe that (assuming $\delta(G_{n,p}) \geq \layers (e-1) $) equation \eqref{eq:dombydeg} gives a bound on the domination number that holds uniformly over all partitions of $G_n$ into a multilayer graph $\mathcal{G}$ and only depends on the minimum degree $\delta(G_{n,p})$ of the underlying random graph. Thus, to bound the extremal cop number using \eqref{eq:dombydeg} we just need to bound the minimum degree of $G_{n,p}$. The degree of each vertex in $G_{n,p}$ is distributed as a sum (denoted by $X$) of $n-1$ independent Bernoulli trials, each with success probability $p$. Thus, as $np \geq  n^{1/2 + \eps}$ and $ n/2 \leq 3(n-1)/4$ for large $n$, by Lemma \ref{lem:chb} and the union bound   
		\begin{equation}\label{eq:probofdeg}\mathbb{P}\left(\delta(G_{n,p}) \leq \frac{1}{2} np  \right) \leq n\cdot \mathbb{P}\left(|X - (n-1)p| \geq  \frac{1}{4}(n-1)p  \right) \leq n\cdot e^{-(\tfrac{1}{4})^2 \cdot \tfrac{(n-1)p}{3}} = o(e^{-\sqrt{n}}).  \end{equation}
		Recall that $\layers \leq n^{\eps}$ by hypothesis, and $n\geq \delta(G_{n,p}) \geq  np/2 = \omega(\layers) $ with probability $1-o(e^{-\sqrt{n}}) $ by \eqref{eq:probofdeg}. The upper bound in the statement follows by applying \eqref{eq:dombydeg} as with probability at least $1- o(e^{-\sqrt{n}})$, this gives \[\gamma(\mlg)\leq   \frac{n\layers}{\layers + \delta(G_{n,p})}\cdot  \left(  \ln\left(\frac{\layers + \delta(G_{n,p})}{\layers}\right)+ 1 \right)\leq   \frac{n\layers}{ np/2}\cdot  \left(  \ln\left(\frac{n^{\eps}+n}{1}\right)+ 1 \right) \leq 10 \cdot \frac{ \layers     \ln  n}{p} ,\] which establishes the upper bound.

		We now prove the lower bound. The rough idea will be to construct a random multilayer graph, lower bound the multilayer cop number of $\mlg$ by showing it is $(1,k)$-m.e.c.\ for a specified $k$, and finally couple this with $G_{n,p}$ to give the desired lower bound. 
		
		In order to begin with the construction of $\mathcal{G}$ we must first (implicitly) define a function $p^*$ of $p$ and $\layers$. For any given $p\in[0,1]$ and $\layers \geq 1$, by the intermediate value theorem there exists some $p^*=p^*(\layers, p)$ satisfying \begin{equation}\label{eq:defofp*} 1-(1-p^*/\layers)^\layers = p.\end{equation} We now claim that for any $0\leq p\leq 1/2$ and $\layers \geq 1$, the corresponding $p^*$ given by \eqref{eq:defofp*} satisfies   \begin{equation}\label{eq:pvsp*} \frac{p^*}{2} \leq p \leq p^*,\quad \text{and consequently,} \quad 0\leq p^*\leq 1 . \end{equation}To begin note that by rearranging \eqref{eq:defofp*} we have $ p^* = \layers\cdot (1- (1-p)^{1/\layers} ) < \layers$. Thus, by Bernoulli's inequality \[p=1 -(1-p^*/\layers)^\layers\leq \layers\cdot p^*/\layers = p^*.\] On the other hand, observe that $p = 1- \left(1-p^*/\layers\right)^\layers\geq1- e^{-p^*} $ since $1-x\leq e^{-x} $ for all $x\geq 0$. By rearranging and taking the logarithm we have $p^*\leq -\ln(1-p) < \frac{p}{1-p} \leq 2p $ as $p\leq 1/2$, giving \eqref{eq:pvsp*}.

		We will construct a multilayer graph $\mlg= (V,\{C_1, \dots, C_\layers\}, *)$ by sampling graphs $G_i $ independently from $G_{n, p^*/\layers}$ and setting  $C_i=E(G_i)$ for each $i$.  Our next step is to show that  $\mathcal{G}$  is $(1, k)$-m.e.c.\ for some $k$ we now set. Recall that $np\geq n^{1/2+\eps}$ for some fixed $\eps>0$, where $\eps \leq 1/2$ as $p\leq 1$, and set 
		\begin{equation}\label{eq:constc}k=  \left\lfloor - \frac{\eps}{2} \cdot \frac{\ln n }{ \ln (1- p^*/\layers) }\right\rfloor.\end{equation}  
		Using the fact that $\frac{x}{1+x}\leq \ln(1+x)\leq x$ for any $x>-1$ and the bound $p \leq p^*$ from \eqref{eq:pvsp*} yields 
		\begin{equation}\label{eq:upbddonk}k \leq - \frac{\eps}{2} \cdot \frac{\ln n }{ \ln (1- p^*/\layers) }  \leq \frac{\eps}{2} \cdot \frac{\layers \ln n }{ p^*}\leq  \frac{\layers \ln n }{ p^*} \leq    \frac{\layers \ln n }{ p}.\end{equation} For sufficiently large $n$ we can assume $p^*\leq 1/3$, so the bound $p^* \leq 2p$ in \eqref{eq:pvsp*} gives 
		\begin{equation}\label{eq:lowbddonk}k \geq - \frac{\eps}{2} \cdot \frac{\ln n }{ \ln (1- p^*/\layers) } -1 \geq   \frac{\eps}{2} \cdot \frac{(1-p^*/\layers) \ln n  }{  p^*/\layers  }-1\geq \frac{\eps}{3} \cdot \frac{\layers \ln n }{ p^*}- 1 \geq \frac{\eps}{10} \cdot \frac{\layers \ln n }{ p}.\end{equation}

		Let $X$ be the number of collections of sets $S_1, \dots, S_\layers \subseteq V$, with $\sum_{i\in [\layers]}|S_i|=k$, and vertices $v \not\in S_1\cup \cdots \cup S_\layers$ such that there does not exist any vertex $x\notin \{v\}\cup S_1\cup \cdots \cup S_\layers$ where $xv\in R$ and $s_ix\notin C_i$ for any $i\in [\layers]$ and $s_i\in S_i$. Observe that if $X=0$ then $\mathcal{G}$ is $(1, k)$-m.e.c.. 
		
		By independence and the definition of $k$ from \eqref{eq:constc}, for a vertex $x 
		\in V  \backslash (\{v\}\cup S_1\cup \cdots \cup S_\layers )$ the probability that $x$ is adjacent to $v$ but not to any vertex in any of the sets $S_1,\dots, S_\layers$ is given by \begin{equation}\label{eq:prob}p\cdot \left(1 - \frac{ p^*}{\layers}\right)^{|S_1|}\cdots \left(1 - \frac{ p^*}{\layers}\right)^{|S_\layers|} = p\cdot \left(1- \frac{p^*}{\layers}\right)^k \geq p\cdot n^{-\eps/2}.\end{equation} 
		
		Since edges are sampled independently, the probability that no such vertex $x\in V  \backslash (\{v\}\cup S_1\cup \cdots \cup S_\layers )$ can be found for a particular $ S_1,\dots, S_\layers $   and $v$ is \begin{equation}\label{eq:bddonmek} \left(1 - p\left(1-  \frac{p^*}{ \layers} \right)^k\right)^{n-|S_1\cup \cdots \cup S_\layers|-1}\leq \left(1-pn^{-\eps/2}\right)^{n-k-1}\leq e^{-pn^{-\eps/2}\left(n - \frac{\layers\ln n}{p} -1\right)}\leq e^{- \frac{1}{2}\cdot n^{-\eps/2}\cdot np }  , \end{equation} for  large $n$, where we have used \eqref{eq:prob}, $1+x\leq e^x$ for any $x>-1$, \eqref{eq:upbddonk}, and $\layers=o(\frac{np}{\log n})$. Observe that \begin{align*}\mathbb{E}\left[ X\right] &\leq  \sum_{k_1+\cdots +k_\layers=k}\binom{n}{k_1}\cdots\binom{n}{k_\layers} \cdot n \cdot \left(1 - p\cdot \left(1- \frac{p^*}{\layers}\right)^k\right)^{n-|S_1\cup \cdots \cup S_\layers|-1}.\end{align*}Note that there are at most $\layers^k$ items in the sum above, as each of the $k$ cops can choose from at most $\layers$ many layers. Thus, by \eqref{eq:bddonmek} for large $n$,  
		\begin{equation}\label{eq:expX}
			\mathbb{E}\left[ X\right]   \leq   \layers^k \cdot {n}^{k }\cdot  n \cdot e^{- \frac{1}{2}\cdot n^{-\eps/2}\cdot np} \leq  n^{2k+1}e^{- \frac{1}{2}\cdot n^{-\eps/2}\cdot np} \leq  \exp\left(\frac{2\layers(\ln n)^2}{p} + \ln n  -\frac{1}{2} \cdot n^{-\eps/2}\cdot np \right),\end{equation}where we have used the bound \eqref{eq:upbddonk} on $k$. By the conditions $\layers\leq n^\eps$ and $np\geq n^{1/2+\eps}$, we have  \[ \frac{ n^{\eps/2}}{ np }\cdot  \frac{\layers(\ln n)^2}{p}=  \frac{n^{\eps/2} n }{(np)^2}\cdot  \layers(\ln n)^2 \leq \frac{n^{\eps/2} n }{n^{1+2\eps}}\cdot  n^\eps (\ln n)^2 = o(1),  \] and thus, for large $n$,\[\frac{2\layers(\ln n)^2}{p} + \ln n -\frac{1}{2} \cdot n^{-\eps/2}\cdot np  \leq   -\frac{1}{3} \cdot n^{-\eps/2}\cdot np \leq - \frac{1}{3}\cdot n^{(1  + \eps)/2} .\] Hence, by Markov's inequality and \eqref{eq:expX} we have
		\[\mathbb{P}\left(X\geq 1 \right) \leq \mathbb{E}\left[ X\right] \leq \exp\left(- \frac{1}{3}\cdot n^{(1  + \eps)/2}  \right) =  e^{-\Omega(\sqrt{n})}.  \] It follows that with probability  $1-e^{-\Omega(\sqrt{n})}$ the multi-layer graph $\mlg$ is $(1,k)$-m.e.c.\ for $k\geq \frac{\eps}{10}\cdot \frac{\layers \cdot \ln n}{p} $ by \eqref{eq:lowbddonk}. Thus if we define the event $ \mathcal{M}=\{\mcop{\mathcal{G}}\geq \frac{\eps}{10}\cdot \frac{\layers \cdot \ln n}{p}\} $ then applying Lemma \ref{lem:mec} gives  \begin{equation*} \Pr{\neg \mathcal{M}}=  e^{-\Omega(\sqrt{n})}.\end{equation*} 
		Let $\mathcal{C}$ be the event that $G_i$ for each layer $i\in [\layers]$ of $\mathcal{G}$ is connected. As $p^*/\layers \geq p/\layers \geq  1/\sqrt{n}$ by \eqref{eq:pvsp*}, each  $G_i$ is not connected with probability at most  $e^{-\sqrt{n}/3}$ by Lemma \ref{lem:randomgraphconprob}. Thus, by the union bound, we have \begin{equation*}
			\Pr{\neg \, \mathcal{C}} \leq  \layers \cdot e^{-\sqrt{n}/3} \leq  e^{-\sqrt{n}/4}.  
		\end{equation*}

		Recall that for a fixed graph $G$ the $\mathcal{L}(G,\layers)$ contains all multilayer graphs $\{(V, \{C_1, \dots, C_\layers\}, *)$ where $E = C_1 \cup \cdots \cup C_\layers$  and for each $i\in[\layers]$, the layer $(V,C_i)$ is connected. Hence, for each $G\in 2^{\binom{n}{2}}$ there is a corresponding fixed collection $\mathcal{L}(G,\layers) \subseteq (2^{\binom{n}{2}})^{\layers} $. Thus, for any given $\layers\geq 1$, taking $G=G_{n,p} $ induces a probability distribution $\mathfrak{L}(n,p,\layers)$ over subsets of $(2^{\binom{n}{2}})^{\layers}$.

		By our choice of $p^*$, given implicitly by $p =1-(1-p^*/\layers)^{\layers} $ in \eqref{eq:defofp*}, we can couple $\mathcal{G}$ to $G(n,p)$; this holds since $\flatten{\mlg}=\cup_{i\in[\layers]}G_i$ follows the distribution of $G_{n,p}$. This coupling also works in the opposite direction: given $G_{n,p}$ we sample each graph $G_i$ by including each edge $e\in E(G_{n,p})$ with probability $p^*/(p\layers)$, thus each edge $e$ is present in $G_i$ independently with probability $p\cdot p^*/(p\layers) = p^*/ \layers$. Observe also that conditional on $\mathcal{C}$ we can couple $\mathcal{G}$ to a random collection $\mathcal{L}(G_{n,p},\layers)$ with distribution $\mathfrak{L}(n,p,\layers)$ so that $\mathcal{G}\in  \mathcal{L}(G_{n,p},\layers)$. It follows from the union bound that 
		\[\mathbb{P}\left( \maxcop{\layers}{G_{n,p}} <  \frac{\eps }{10} \cdot \frac{ \layers \cdot \ln  n}{ p}\right)  \leq  \Pr{\neg \mathcal{M}} + \Pr{\neg \mathcal{C}} \leq e^{-\Omega(\sqrt{n})} +e^{-\sqrt{n}/4} \leq e^{-\Omega(\sqrt{n})} , \]  
		which gives the lower bound.
	\end{proof}

	\subsection{Treewidth} \label{sec:trw}
	
	A \textit{tree decomposition} \cite[Ch.\ 12]{Diestel} of a graph $G=(V,E)$ is a pair $\left(\{X_i: i\in
	I\},\right.$ $\left.T=(I,F)\right)$ where $\{X_i: i \in I\}$ is a family of
	subsets (or ``bags'') $X_i\subseteq V$ and $T = (I,F)$ is a tree such that
	\begin{itemize}
		\item $\bigcup_{i \in I} X_i = V$,
		\item for every edge $vw \in E$ there exists $i \in I$ with $\{v,w\}
		\subseteq X_i$,
		\item for every $i,j,k \in I$, if $j$ lies on the path
		from $i$ to $k$ in $T$, then $X_i \cap X_k \subseteq X_j$.  
	\end{itemize}
	The \textit{width} of $\left(\{X_i:i \in I\},T=(I,F)\right)$ is defined as
	$\max_{i \in I} |X_i| -1$. The \textit{tree\-width} $\trw(G)$ of $G$ is	the minimum width of any tree decomposition of $G$.
	
	The following result of Joret, Kami\'nski \& Theis bounds the single-layer cop number of a graph by its treewidth.

	\begin{proposition}[{\cite[Proposition 1.5]{JoretKT10}}]The cop number of a graph $G$ is at most $\trw(G)/2 + 1$.  
	\end{proposition}

	Treewidth has a another known relationship with cop number as Seymour and Thomas~\cite{SeymourThomas} use the helicopter variant of cops and robbers and show that if a graph $G$ has cop number $k$ then the treewidth of $G$ is at most $k-1$. We do not use this helicopter variant here.
	
	\begin{proposition}
		For any connected graph $G := (V,E)$, and any integer $\layers \geq 1$, we have $\maxcop{\layers}{G} \leq \trw(G)$.
		Furthermore, these cops can be placed in any layers and still capture the robber.  
	\end{proposition}
	\begin{proof}Let $\mathcal{G}=(V, \{C_1, \dots, C_\layers\}, *)\in \mathcal{L}(G,\layers)$ be any multi-layer graph such that for each $i\in [\layers]$ the graph $(V,C_i)$ is connected, and $C_1 \cup \cdots \cup C_\layers = E$. Let $T=(I,F)$ be an optimal tree decomposition of $G$ with $k=\trw(G)$. To begin choose any bag $X_i$, and assign $k$ cops to the vertices of $X_i$ (in any layers) in any way so that each vertex of $X_i$ receives at least one cop. This ensures that the robber cannot start at any vertex in $X_i$. Hence the robber starts at some vertex not in $X_i$, and they are then confined to the subgraph corresponding to a single tree $T'$ of $T\backslash \{i\}$.

		Let $X_j$ be the bag corresponding to the unique vertex $j \in I$ in $T'$ such that $ij\in F$. We now move the cops one by one from the vertices of $X_i$ to cover the vertices of $X_j$. By \cite[Lemma 12.3.1]{Diestel}, the vertices $X_i\cap X_j$ form a cut-set of $G$ and so as long as we leave at least one cop at each vertex of $X_i\cap X_j$ while we move the remaining cops then the robber cannot leave the tree $T'$. Thus first move all cops (one by one) from vertices in $X_i \backslash (X_i\cap X_j)$ to vertices in $X_j \backslash (X_i\cap X_j)$, assigning at most one cop to a vertex (unless all vertices are covered). If there are still vertices of $X_j$ uncovered then we can find extra cops to cover these vertices from vertices in $X_i\cap X_j$ which have more than one cop. There are always enough cops to do this since, for any $i,j\in I$,  $ |X_j\backslash (X_i\cap X_j)|+ |X_i\cap X_j| = |X_j| \leq k $. Furthermore, since each layer is connected we can always move the cops in this way. 
		
		Thus, we can keep moving the cops to a bag that decreases the size of the tree $T'$ containing the robber until $T'$ just consists of vertices in the bag covered by the cops, thus they have caught the robber.     
	\end{proof}

	\section{Multi-Layer Analogue of Meyniel's Conjecture}\label{sec:men}
	
	For the classical cop number, Meyniel’s Conjecture~\cite{Frankl87} states that  $\mathcal{O}(\sqrt{n})$ cops are sufficient to win cops and robbers on any connected graph $G$. Following earlier results \cite{Chiniforooshan08,Frankl87} the current best bound stands at \[n\cdot 2^{-(1-o(1))\sqrt{\log_2 n}},\] proved by \cite{FriezeKL12,LuP12,ScottS11} independently, see \cite[Chapter 3]{Bonato2011Book} for a more detailed overview. 
	
	It is natural to explore analogues of Meyniel’s Conjecture for the multi-layer cop number. The question we will consider is the minimum number of cops needed to patrol any multi-layer graph with $\layers$ connected layers. That is, we seek a bound on the extremal multi-layer cop number $\maxcop{\layers}{G}$  which holds for all connected graphs $G$. The cycle example from \cref{rem:cycleexample} and the clique result from \cref{thm:complete} show that if layers can be disconnected, or $\layers$ can grow arbitrarily with $n$, then no bound better than $\mathcal{O}(n)$ can hold.  We conjecture that this is not the case when layers are connected and the number $\layers$ of them is bounded. 
	\begin{conjecture}\label{multimeyniel}Let $\layers\geq 1$ be any fixed integer and  $G$ be any $n$-vertex connected graph. Then, we have \[\maxcop{\layers}{G} =o(n).\]   
	\end{conjecture}
	This conjecture might seem very modest in comparison to Meyniel's conjecture, however the following result shows that it would be almost tight, even for the case $\layers =2$ of just two cop layers. 
	\begin{theorem}\label{thm:meynielcounter} For any positive integer $n$ there is an $n$-vertex multi-layer graph $\mlg=(V, \{C_1,C_2\}, *)$ such that $|V|=\Theta(n)$, each cop layer is connected and has bounded cop-number, yet \[\mcop{\mathcal{G}} = \Omega\left(\frac{n}{\log n}\right).\] 
	\end{theorem}

	Many of the current approaches to Meyniel’s Conjecture use some variation of the fact that a single cop can guard any shortest path between any two vertices. For example, the first step of the approach in \cite{ScottS11} is to iteratively remove long geodesics until the graph has small diameter (following this a more sophisticated argument matching randomly placed cops to possible robber trajectories is applied). What makes Conjecture \ref{multimeyniel} difficult to approach is that, even for two layers, a shortest path in the flattened graph $\flatten{\mathcal{G}}$ may not live within a single cop layer. We note that \cite{FriezeKL12} uses a different approach based on expansion, making it more versatile,  however the authors were unable to apply it in the multi-layer setting. We believe a new approach is needed, hopefully progress on this problem might also give an insight into the classical Mayniel conjecture. 
	
	Nevertheless, using a simple dominating set approach we can at least prove Conjecture \ref{multimeyniel} for multi-layer graphs with diverging minimum degree. 
	
	\begin{proposition}For any connected $n$-vertex graph $G$ and $\layers:=\layers(n)\geq 1$ satisfying $\delta(G) / \layers  \rightarrow \infty $,  we have \[\maxcop{\layers}{G} =o(n).\]
	\end{proposition}
	\begin{proof} If $\mathcal{G}$ is any multi-layer graph such that $\flatten{\mlg}=G$ then $\delta(\mlg)/\layers \geq \delta(G)/\layers \rightarrow \infty$ by \cref{prop:multmin>flatmin} and hypothesis. The result then follows by  Lemma~\ref{lem:copdom} since $\ln(x)/x\rightarrow 0 $ as $x\rightarrow \infty$.  
	\end{proof}
	
	Note that in the proof of this result we did not actually need the graph/layers to be connected. 
	
	\subsection*{Proof of Theorem \ref{thm:meynielcounter}} To begin, we introduce the \textit{cops-bane} family of multi-layer graphs, denoted by $\mathcal{CB}(\delta,N)$ for $\delta\in(0,1)$ and $N\geq 1$.
	We show that, for suitable $\delta,N$, the family $\mathcal{CB}(\delta,N)$ is non-empty and any $\mathcal{G} \in \mathcal{CB}(\delta,N)$ requires $\Omega(N)$ cops to guard, this establishes Theorem \ref{thm:meynielcounter} since each graph in $\mathcal{CB}(\delta,N)$ has $\Theta(N\log N)$ vertices.

	Before we define cop-bane multi-layer graphs we will introduce some graph theoretic notions used in the construction and proof, in particular we will make use of clustered colourings \cite{WoodSurvey}. Given an assignment of colours to the vertices (respectively edges) of a graph $G$, a monochromatic component is a connected component of the subgraph induced by all the vertices (resp.\ edges) assigned to a single colour. A graph $G$ is $k$-vertex-colourable (resp. $k$-edge-colourable) with clustering $c$ if each vertex (resp.\ edge) can be assigned one of $k$ colours such that each monochromatic component has at most $c$ vertices.  
	
	\begin{lemma}[c.f.\ \cite{Haxell}]\label{lem:hax}
		Any graph of maximum degree at most $3$ has a  $2$-edge-colouring with clustering $2000$. 
	\end{lemma}
	\begin{proof}We begin by constructing the line graph of $G$, which is the graph $H$ with $V(H)=E(G)$ and $E(H)=\{xy: x \text{ and }y \text{ share an endpoint in }G \}.$ Observe that $H$ has maximum degree at most $4$ since $G$ has maximum degree at most $3$. Haxell, Szab\'o and Tardos \cite{Haxell} showed that every graph of maximum degree at most $5$ has a $2$-vertex-colouring with clustering at most $2000$. We take such a vertex colouring of $H$, which translates to a $2$-edge-colouring of $G$ with the same clustering.  
	\end{proof}
	For $0\leq \alpha\leq 1 $, a graph $G=(V,E)$ is an \textit{$\alpha$-vertex-expander} if for every $S \subseteq V$ of size at most $|V|/2$ there are at least $\alpha|S|$ vertices of $ V \backslash S$ connected by an edge to $S$, that is, \[\min_{S\subseteq V : |S|\leq n/2}\frac{|\{v\in V\backslash S : u\in S, uv\in E\}|}{|S|}\geq  \alpha .\]
	
	We make use of the following well known fact. 
	
	\begin{lemma}[{\cite[Page 455]{ExpanderSurvey}}]\label{lem:expdiam} For any constant $0<\alpha\leq 1$ there exists a constant $C_{\mathsf{diam}}(\alpha)$  such that any $n$-vertex $\alpha$-vertex-expander is connected and has diameter at most $C_{\mathsf{diam}}(\alpha)\cdot \log n$.
	\end{lemma}

	For a constant $\alpha>0$ and even integer $N\geq 2$ we define the cops-bane family $\mathcal{CB}(\alpha,N)$ to be all $N$-vertex multi-layer graphs $\mathcal{G}$ with robber layer $R$ and cop layers $C_1$ and $C_2$ generated as follows.
	Let $X=(V_X,E_1\cup E_2)$ be a $3$-regular $\alpha$-vertex-expander on $N$ vertices where $E_1$ and $E_2$ form the colour classes of a $2$-edge-colouring with clustering at most $2000$. Let $ D=\lceil C_{\mathsf{diam}}(\alpha/2)\cdot \log N\rceil  $ where $C_{\mathsf{diam}}(\alpha/2)$ is the constant defined in Lemma \ref{lem:expdiam}. Additionally let $S=(V_S,E_S)$ be a star on $N+1$ vertices where we label the centre vertex of the star as $\infty$, and each edge in the star has been replaced by a path of $2D+1$ edges, we call such a path an \textit{arm}. We identify each leaf of $S$ with a unique element of $V_X$. We then set $C_1 = E_1 \cup E_S$ and $C_2 = E_2 \cup E_S$ as the two cop layers. Finally, let $R=E_1 \cup E_2 $, and then $\mlg=(V_X \cup V_S, \{C_1, C_2\}, R)$, where $|V_X \cup V_S| =\Theta(N \log N)$. See Figure \ref{fig:copsbane} for an example of a cops-bane multilayer graph. 
	
	\begin{figure}
		\centering
		
		\begin{tikzpicture}[xscale=.5,yscale=.5, node/.style={thick,circle,draw=black,minimum size=.1cm,fill=white},dotnode/.style={thick,dotted,circle,draw=black,minimum size=.1cm,fill=white},edge0/.style={line width=3pt,col0}, edge1/.style={line width=3pt,col1},edge2/.style={line width=3pt,col2},edge3/.style={line width=3pt,col3},edge4/.style={line width=3pt,col4},edge5/.style={line width=3pt,col5},edge6/.style={line width=3pt,col6},edgea/.style={line width=3pt,cola},edgeb/.style={line width=3pt,colb},edgec/.style={line width=3pt,colc},edgeS/.style={line width=3pt,col1,dotted}]
			
			\def \radius {3cm}
			\foreach \s in {0,...,7}
			{
				\node[node] (\s) at ({360/8 * (- \s +  2 ) -360/16 }:\radius)  { };;
			}
			
			\node[node] (s) at ({180}:3*\radius)  {};;
			
			\draw[edgec] (7) to (0);
			\draw[edgea] (1) to (2);
			\draw[edgec] (3) to (4);
			\draw[edgea] (5) to (6);
			
			\draw[edgea] (0) to (1);
			\draw[edgea] (2) to (3);
			\draw[edgea] (4) to (5);
			\draw[edgea] (6) to (7);   	
			
			\draw[edgec] (0) to (2);
			\draw[edgec] (1) to (4);
			\draw[edgec] (5) to (7);
			\draw[edgec] (6) to (3);	 	
			
			\draw[edgeS] (s) to (4);
			\draw[edgeS] (s) to (7);
			\draw[edgeS] (s) to (6);
			\draw[edgeS] (s) to (5);  
			\draw[edgeS] (s) to[out=30,in=165] (0);
			\draw[edgeS] (s) to[out=50,in=110] (1); 	 		\draw[edgeS] (s) to[out=-50,in=-110] (2); 	
			\draw[edgeS] (s) to[out=-30,in=-165] (3); 	 		    	\end{tikzpicture}\vspace{-15pt}
		\caption{A cops-bane multi-layer graph from the family $\mathcal{CB}(3/4,8)$. The dotted yellow edges represent paths of length $2D+1$, where $D=\lceil C_{\mathsf{diam}}(3/8)\cdot \log 8\rceil  $ and $C_{\mathsf{diam}}(3/8)$ is from Lemma~\ref{lem:expdiam}. The robber layer contains all edges, then each of the two cop layers consists of the orange dotted edges plus all of either the red or blue edges. The graph with just orange dotted edges is $S$, the graphs with just red and blue edges is the $3/4$-vertex-expander $X$. The red and blue edges are a $2$-edge-coloring of $X$ with clustering $3$.}\label{fig:copsbane}
		
	\end{figure}
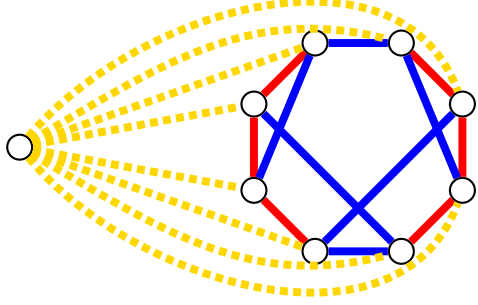
	\begin{lemma}\label{prob:Existscopsbane} There exists a constant $0<\alpha\leq 1$ such that for any sufficiently large even integer $N \geq 4$ the family $\mathcal{CB}(\alpha,N)$  is non-empty.
	\end{lemma} 
	\begin{proof} Kolesnik and Wormald \cite{KW} show that there exists a fixed $\alpha>0$ such that a random $3$-regular graph is an $\alpha$-vertex-expander with high probability. Thus, for any suitably large $N$ there exists a $3$-regular $N$-vertex $\alpha$-vertex-expander. By \cref{lem:hax} there is a $2$ colouring with clustering $2000$.  
	\end{proof}
	
	We now show that, owing to their simple structure, the cop layers are easy to police. 
	
	\begin{lemma}For any $(V_X\cup V_S, \{C_1,C_2\}, R)\in \mathcal{CB}(\alpha,N)$  we have $\cop{(V,C_1)}, \cop{(V,C_2)} \leq 2001$. 
	\end{lemma} 
	\begin{proof}
		Recall that, for each $i\in \{0,1\}$, we have  $C_i=E_i\cup E_S$ where $E_i\cap E_S= \emptyset$. Thus, for $i\in \{0,1\}$, $E_i$ induces (in $X$) a set of disjoint connected components (subgraphs) $x_1, \dots, x_{k_i}$ each on at most $2001$ vertices, and $E_S$ induces a star whose edges are subdivided $2D+1$ times, and where each leaf is identified with a unique vertex in $V_X$. We now describe the strategy for $2001$ cops to catch the robber on $(V,C_i)$. 
		
		To begin we place all $2001$ cops at the centre $\infty$ of the star $S$. To avoid capture in the first round, the robber places themselves at a vertex $u$ which cannot be $\infty$ or a neighbour of $\infty$. Now, either $u$ is on an arm of $S$ that does not lead to any component induced by the clustered colouring, or one that does (in this case call it $x_j$). In the first case we simply send a single cop along the arm. In the second case we choose $|V(x_j)|\leq 2001$ many cops and move each along a unique arm with endpoint in $V(x_j)$ one step at a time toward this endpoint. Thus, in either case the robber is caught in at most $2D+1$ steps.\end{proof}
	
	We will use the following theorem which shows that expanders are resistant to vertex deletions.  
	
	\begin{theorem}[{Rephrasing of \cite[Theorem 2.1]{BagchiBCES06}}]\label{thm:prune} Let $G$ be an $n$-vertex $\alpha$-vertex-expander of maximum degree $\Delta$, and let $C > 1$ be a  fixed constant. After the removal of any set of up to $f =\alpha n/(4\Delta C^2)$ vertices from $V(G)$, the resulting graph contains a subgraph $H$ on at least $n -   f \cdot (C /\alpha)$ vertices with
		vertex-expansion at least $(1 - 1/C) \cdot \alpha$.
	\end{theorem}
	
	We now use the above results to show that $\alpha N/10^6$ cops cannot catch the robber in any small part of the expander $X$ within a cops-bane graph. 
	
	\begin{lemma}\label{lem:copsbane} For any constant $0<\alpha\leq 1$ and sufficiently large  natural number $N$, each multi-layer graph in the family $\mathcal{CB}(\alpha, N)$ has multi-layer cop number at least $\alpha N/10^6$. 
	\end{lemma}
	
	\begin{proof} The robber will only move on the edges of the $\alpha$-vertex-expander $X$. 
		We say that a vertex $v\in V_X$ is \emph{blocked} if there exists some cop $c$ that can reach $v$ without visiting $\infty$, the vertex at the centre of the star.
		For clarity, if a cop is on $\infty$ then it does not block any vertex.
		Note that any cop can block at most $2001$ vertices, since the largest monochromatic component in the $2$-edge-colouring of $X$ has at most $2000$ edges.
		The rough idea for our robber will be to move between sets of unblocked vertices.
		For a set of (blocked) vertices $B$, if there exists a subset of $V_X\setminus B$ with size at least $N/2+1$ which induces a subgraph $H\subseteq G$ with diameter at most $D$ then we call $H$ \textit{safe} from $B$. By the lower bound on the number of vertices in such a safe subgraph, any two safe subgraphs must intersect in at least one vertex.  
		
		Now define integers $f$ and $k$ as follows 
		\begin{equation*}
			\label{eq:noofcops}k= \left\lfloor \frac{f}{2001}\right\rfloor   ,\qquad \text{where}\qquad f= \left\lfloor\frac{\alpha N}{48}\right\rfloor.
		\end{equation*}
		We will show that $k$ cops cannot catch a robber; we use $f$ to denote the maximum number of vertices of $X$ that $k$ cops can block.
		
		\begin{claim}\label{clm:safegraph}
			For any set $B\subseteq V$ satisfying $|B|\leq f$ there exists a subgraph $H \subseteq G$ which is safe from $B$. 
		\end{claim}
		
		\begin{poc}
			
			Observe that, since $X$ is $3$-regular, our choice of $f$ corresponds to setting $C=2$ in Theorem \ref{thm:prune} and choosing at most $f$ vertices adversarially. Hence, by Theorem \ref{thm:prune}, removing the set $B$ from $V(X)$ leaves a connected subgraph $H$ on at least 
			\begin{equation*} N - f\cdot\frac{ 2 }{\alpha} \geq  N - \frac{\alpha N}{48}\cdot\frac{ 2 }{\alpha} > \frac{N}{2} \end{equation*} vertices, and vertex-expansion at least $\alpha/2>0$. Hence by Lemma \ref{lem:expdiam} the diameter of $H$ is at most $C_{\mathsf{diam}}(\alpha/2)\cdot \log N =D$, thus $H$ is safe from $B$. 
		\end{poc}
		
		The robber's strategy will be to occupy a vertex in $V_X$ that is safe from the current set of blocked vertices at each step. It does this by trying to anticipate where the cops are heading next (the long paths in $S$ give them such foresight), computing a safe subgraph based on this prediction that will remain safe for at least as long as it takes to reach this subgraph, compute a new safe subgraph (if the current one became threatened), and then move to the new safe subgraph. Claim \ref{clm:safegraph} guarantees that the required safe subgraphs exist in each step. 
		
		We will say that the \emph{cop location restrictions} hold, if, on a given turn the robber is on some vertex within a subgraph $H'$ of $X$ with at least $N/2+1$ vertices such that $H'$ has diameter at most $D$, and no cop is within $D$ turns of reaching any vertex in $H'$.
		
		At the start of the game of cops and robbers, the cops are placed on some set of vertices, and let $B$ be the set of vertices that are blocked by these cops.
		Then by Claim~\ref{clm:safegraph} there is subgraph $H$ of $X$ that is safe from $B$, so the robber picks an arbitrary vertex in $H$ to start on.
		We easily see that the cop location restrictions are satisfied by this, and will be satisfied after the cops have one turn to move.
		
		Now, assume that it is the robbers turn, and that the cop location restrictions hold. Let $H'$ denote the subgraph of $X$ from the cop location restrictions.
		Let $B$ be the set of vertices that are currently blocked.
		By Claim~\ref{clm:safegraph} there is a subgraph $H$ of $X$ that is safe from $B$.
		As $H$ and $H'$ both have at least $N/2 + 1$ vertices, there must be at least one vertex $v\in V_X$ such that $v$ is in both $H$ and $H'$.
		Additionally, the diameter of $H'$ is at most $D$, so the robber can move along a shortest path in $H'$ to $v$ in at most $D$ turns without leaving $H'$.
		As the cop restrictions held before these $D$ turns, the robber cannot be caught during these moves.
		As $v$ is also in $H$, and as there was no cop that was blocking any vertex in $H$ before the robber started these moves,
		any cop that is now blocking some vertex in $H$ must have had to move to (or start on) $\infty$, and then move from $\infty$ towards some vertex in the expander.
		However, the distance from $\infty$ to any vertex in the expander is $2D+1$, so 
		after $D$ moves no cop can be within a distance of $D$ to any vertex in $H$.
		The cop location restrictions are therefore satisfied after these $D$ turns,
		and the robber can repeat this process indefinitely to win.
	\end{proof} 
	\section{Conclusion and Open Problems} \label{sec:conclusion} 
	
	We studied the game of cops and robbers on multi-layer graphs via several different approaches, including concrete strategies for certain graph, the construction of counter-intuitive examples, algorithmic and hardness results, and the use of probabilistic methods and expanders for extremal constructions.  We find that the multi-layer cop number cannot be bound from above or below by (non-constant) functions of the cop numbers of the individual layers. We also introduce an extremal variant of the problem and obtain asymptotically tight bounds for cliques and dense binomial random graphs (extending some tools from the single layer case along the way). We also find that a naive transfer of Meyniel's conjecture to the multi-layer setting is not true: there are $n$-vertex multi-layer graphs with two cop layers which require $\Omega( n / \log n)$ cops to police. Algorithmically, we find that even if each layer is one tree with some isolated vertices, the free layer choice variant of the problem remains NP-hard. Positively, we find that the problem can be resolved by an algorithm that is FPT in the number of cops and layers if the layer the robber resides in is a tree.

	We are hopeful that our contribution will spark future work in multi-layer variants of cops and robbers, and suggest a number of possible open questions:

	There are some frequently used ideas from single-layer cops and robbers which we did not utilise. For example, we have no useful notion of a corner, or a retract, nor dismantleability - it is possible that such tools can be extended to the multi-layer setting.
	
	Some progress was made on the parameterised complexity of our problems, but we have only considered a limited set of parameters and have not explored any parameter that constrains the nature of interaction between the layers. For example, if we require that the layers are very similar alongside other restrictions does that impact the computational complexity of our problems?
	
	Single-layer cops and robbers has been very successful as a tool for defining useful graph parameters of simple graphs. We ask whether multi-layer cops and robbers could be used to define algorithmically useful graph parameters.

	While we showed that a naive adaptation of Meyniel's conjecture to our multi-layer setting does not hold, it is still possible that $o(|V|)$ cops are sufficient for a bounded number of connected layers.  We have shown this for a special case of growing minimum degree: is it true in general?

	Finally, while we introduced a particular notion of multi-layer dominating set for proving upper bounds (inspired by similar ideas in single-layer cops and robbers), this multi-layer graph characteristic may also be interesting in its own right; in particular for algorithms for other problems on multi-layer graphs.  
	
	\section*{Acknowledgements}
	We thank David R.~Wood for a brief but informative discussion on clustered colourings. All authors of this work were supported by EPSRC project EP/T004878/1: Multilayer Algorithmics to Leverage Graph Structure at the University of Glasgow.

\end{document}